\documentclass{article}
\usepackage[utf8]{inputenc}
\usepackage{amsmath,amsfonts,amssymb,latexsym,amsthm}
\usepackage{graphicx}
\usepackage{tikz}
\usetikzlibrary{patterns,decorations.pathmorphing}

\usepackage{hyperref}
\usepackage{tikz}
\usepackage{pgfplots}
\pgfplotsset{compat = newest}
\usepackage{tikz-3dplot}
\usepackage{capt-of}
\usepackage{caption}
\usepackage{comment}
\usepackage{authblk}
\usepackage{orcidlink}

\usepackage[a4paper,
            bindingoffset=0.2in,
            left=1in,
            right=1in,
            top=1in,
            bottom=1in,
            footskip=.25in]{geometry}

\usepackage{blindtext}

\newcommand{\orcid}[1]{\href{https://orcid.org/#1}{\textcolor[HTML]{A6CE39}{\aiOrcid}}}

\tdplotsetmaincoords{60}{120}
\providecommand{\keywords}[1]{\textbf{\textit{Keywords:}} #1}

\newcommand{\de}{\mathrm{d}}

\newcommand{\dof}{\mathrm{dof}}
\newcommand{\R}{\mathbb{R}}
\newcommand{\Phat}{\widehat{\mathcal{P}}}
\newcommand{\Ihat}{\widehat{I}}
\newcommand{\That}{\widehat{T}}
\newcommand{\LTG}{\mathrm{LTG}}
\newcommand{\loc}{\text{loc}} 
\newcommand{\Nloc}{N_{\text{\tiny{loc}}}}

\newtheorem{theorem}{Theorem}[section]
\newtheorem{lemma}[theorem]{Lemma}
\newtheorem{remark}[theorem]{Remark}
\newtheorem{proposition}[theorem]{Proposition}
\newtheorem{definition}[theorem]{Definition}

\author[a]{Ludovico Bruni Bruno\footnote{bruni@math.unipd.it} \orcidlink{0000-0002-5246-8049}}
\author[b]{Matteo Semplice\footnote{matteo.semplice@uninsubria.it}\orcidlink{0000-0002-2398-0828}}
\author[b,c]{Stefano Serra-Capizzano\footnote{s.serracapizzano@uninsubria.it, stefano.serra@it.uu.se}\orcidlink{0000-0001-9477-109X}}
\affil[a]{Department of Mathematics \lq\lq Tullio Levi-Civita\rq\rq, University of Padova, Via Trieste 63, Padova, Italy}
\affil[b]{Department of Science and High Technology, Division of Mathematics,  University of Insubria}
\affil[c]{Department of Information Technology, Division of Scientific Computing, Uppsala University}

\date{}
\title{The numerical linear algebra of weights: from the spectral analysis to conditioning and preconditioning in the Laplacian case}

\begin{document}
\maketitle

\begin{abstract}
Weights are geometrical degrees of freedom that allow to generalise Lagrangian finite elements. They are defined through integrals over specific supports, well understood in terms of differential forms and integration, and lie within the framework of finite element exterior calculus. In this work we exploit this formalism with the target of identifying supports that are appealing for a finite element approximation.
To do so, we study the related parametric matrix-sequences, with the matrix order tending to infinity as the mesh size tends to zero. We describe the conditioning and the spectral global behavior in terms of the standard Toeplitz machinery and GLT theory, leading to the identification of the optimal choices for weights. Moreover, we propose and test ad hoc preconditioners, in dependence of the discretization parameters and in connection with conjugate gradient method.
The model problem we consider is a onedimensional Laplacian, both with constant and non constant coefficients.
Numerical visualizations and experimental tests are reported and critically discussed, demonstrating the advantages of weights-induced bases over standard Lagrangian ones. Open problems and future steps are listed in the conclusive section, especially regarding the multidimensional case.
\end{abstract}

\keywords{Exterior calculus, finite elements, weight selection, spectral distribution of matrix-sequences, conditioning, preconditioning}

\par{\bf AMS subject classifications} 65N30, 15A18, 47B35, 65F10, 65F08

\section{Introduction}\label{sec:intro}

Discretisations of (differential) operators usually yield very large matrices. Performing valuable spectral analyses is thus across-the-board to all numerical methods \cite{opiccolo,tom-paper,GSCS,Rahla}. These analyses can also be used to select optimal parameters, as done in the onedimensional \cite{sixtine1} and multidimensional \cite{sixtine2} hyperbolic framework, where such an analysis has been recently used to investigate stabilization parameters and CFL numbers. The principal aim of this work is to carry a spectral analysis of the generalisation of Lagrangian finite element method induced by weights \cite{BruniThesis} with the target of detecting optimal bases.

One of the main ingredients of a finite element method (FEM) is the selection of degrees of freedom (\emph{dofs}). In fact, \emph{dofs} are required in two key steps of the approximation of a partial differential equation (PDE) via such a method: they prescribe gluing conditions between local solutions on each element and they specify, by duality, the basis of shape functions with respect to which the global solution is expanded \cite{ErnI}.

To begin with, consider an open bounded set $ \Omega \subset \mathbb{R}^d$, $d\ge 1$. 
Recall that the standard approach of a finite element method consists in converting the differential equation
$$ \begin{cases}
  Lu = f \qquad \text{in } \Omega, \\
  \text{boundary conditions on } \partial \Omega,
\end{cases}$$
in the weak formulation
\begin{equation} \label{eq:bilinearfrom}
  a(u,v) = l(f) .
\end{equation}
The domain $ \Omega $ is splitted into or approximated by a mesh $ \mathcal{T} $ 
of $n$ elements such that $ \bigcup_{T \in \mathcal{T}} T = \Omega $ or $ \bigcup_{T \in \mathcal{T}} T \approx \Omega $ in a certain measure, usually the Peano-Jordan one. Then the solution $ u $ is sought in some finite dimensional space $ X $, hence the bilinear form $ a $ may be represented as a finite dimensional matrix, whose entries are the values of the bilinear form computed on a basis of $ X $.

A crucial aspect of finite elements consists in choosing bases of the space $ X $ that are made of functions with \emph{small support}. In particular, shape functions are supported \emph{only on one element $ T $ or few neighbouring elements}, so that the matrix $ A $ representing  $ a: X \times X \to \mathbb{R} $, written with respect to some basis $ \{\varphi_i \}_{i=1}^N $ of $ X $, $ N \doteq \dim X$, namely
$$ A_{i,j} \doteq a (\varphi_i, \varphi_j) ,$$
is sparse and typically sparse in a $d$-level sense if $ \Omega \subset \mathbb{R}^d$, with $d>1$ \cite{Rahla}.


While the choice of the space $X$ controls the Galerkin approximation error,
the choice of a basis for $ X $ heavily affects the conditioning of the matrix $ A $: see the discussion in Section \ref{sec:pb} and a visualization in Figure \ref{fig:optdeg34}. Hence this choice is as important as the approximation features, for mostly two reasons:
\begin{itemize}
  \item the speed of convergence of an iterative solver is directly affected by the conditioning, in a way depending often on the method itself;
  \item a convenient choice of basis may make it easier to construct an effective preconditioner for $A$.
\end{itemize}
%
%
%
%
%

In this paper we aim at studying and optimizing the use of \emph{weights} as degrees of freedom for FEM, detecting the effect on the induced basis for $ X $. Weights are degrees of freedom designed for having an evident geometrical and physical meaning. They were introduced in the context of computational electromagnetism \cite{BossavitBook} and found a relevant role in finite element exterior calculus (FEEC) thanks to their natural interpretation as integrals of differential forms \cite{m2an,RB09}. Throughout the years, they consolidated their role in approximation theory \cite{BE23,BruniRunge} and the corresponding interpolator \cite{nostropreprint,ABRICO} has been reconsidered for FEM approximations \cite{Zappon}. While these dofs coincide with moments  on the lowest degree \cite{Nedelec80,Nedelec86}, for $ r > 1 $ the two approaches sensibly diverge, as the choice of weights turns out simply to be the selection of the supports of integration \cite{BruniThesis}. If, on the one hand, in the context of interpolation it is easily shown that the placement of nodes, edges, faces seriously affects the corresponding projector (usually analysed via the generalised Lebesgue constant \cite{AnaFra}), it is harder to detect the role of this selection in FEM. For this reason, in literature uniform distributions are mostly taken into account \cite{m2an,Zappon}.

As usual in FEM, the global system matrix can be assembled from small ``local matrices'' of the size of $\Nloc\times\Nloc$, where $\Nloc$ is the number of dof's per element. 
In the context of FEM built on $n$ affinely equivalent elements, these local contributions coming from each element will all have the same structure and  the corresponding system matrix $A_n$ will end up having a (generalized) Toeplitz structure. The spectral properties of the matrix sequence $\{A_n\}$ obtained when  increasing the number of elements can thus be analysed following a neat and linear strategy driven by the Toeplitz machinery \cite{BS} and the GLT theory \cite{GLTbookI,etnaGLTbookIII}, which relates them to linear algebraic properties of the small building blocks.
In this way studying the small building block attached to the reference element is enough to study and optimize the basis for the entire FEM space $X$.

In particular, weights-based continuous finite elements in one space dimension will have a degree of freedom located at each endpoint and the others dof's shall be placed inside. Obvious symmetry considerations suggest that the placements of weights inside the element can be described by a single parameter for local spaces $\mathbb{P}_r$ with $r=3,4$, on two parameters for $r=5,6$, etc.
Symmetry sets a remarkable difference between the onedimensional case, which we shall treat here, and the multidimensional one, which requires non-symmetrical choices of weights to obtain unisolvent sets \cite{BFZ23}.
We exploit this simplification to study the one-parameter family of weights based continous FEM through the GLT machinery, leading to a proposal for the weigths location that minimizes the conditioning of $A$ and that allows the construction of optimal and robust preconditioners for the conjugate gradient method. 
Although this analysis sensibly diverges from that usually performed on weights for interpolation theory, it is worth pointing out that optimality and quasi-optimality results fall somewhat close to each other.

The paper is organized as follows. 
In Section \ref{sect:bases} we introduce the formalism of FEEC, the corresponding bases, the notion of weights dof's, their duality and the weak formulation of our test cases.
Section \ref{sec:spectral-tools} contains a short account on the notion of spectral distribution in the Weyl sense, the definition of block Toeplitz and block circulant structures, the $*$-algebra of GLT matrix-sequences, and their main operative features. 
Section \ref{sec:spectral} is devoted to deduce the spectral symbol and the key spectral characteristics of weight-based FEM, using the tools given in the previous two sections.
As a further step, in Section \ref{sec:pb}, we specify degrees 3 and 4. Since these cases are associated with a one-parameter family of GLT sequences, we carry an explicit spectral analysis and the consequent optimization, together with a proposal of a preconditioning strategy.
Theoretical findings are supported by several visualizations and numerical experiments of Section \ref{sec:numer}. Finally, in Section \ref{sec:end}, we give a summary of the present work and we indicate future liness of research opened by the current investigation, especially in the direction of multidimensional problems.

Throughout the paper, we assume for granted basic facts of differential geometry and discrete differential geometry. The interested reader is addressed, for instance, to \cite{AMRbook} and \cite{DEC}. We are not coming across technical aspects of these topics: as a consequence, descriptions in terms of functions and fields provide a sufficiently evocative description already at a first read.

\section{Weights-based FEM} \label{sect:bases}

In the framework of \emph{finite element exterior calculus} \cite{AFW06}, shape functions are thought of as differential forms $ \Lambda^k (\Omega) $ \cite{Warner}, namely sections of the $k$-th exterior power of the cotangent bundle of $ \Omega $. Thus, $ 0$- and $n$-forms (with $ n = \dim \Omega $) are functions, $1$- and $ (n-1)$-forms are (dual to) vector fields, and so on. 
An overview of possible local spaces, together with their assembled counterparts, is offered in the \emph{Periodic Table of finite elements} \cite{PeriodicTable}. Among them, we recall Lagrangian elements, N\'ed\'elec elements, Raviart-Thomas elements, and discontinuous elements.

Regardless of the local basis, the construction of the bases for the FEM discretization space is generally performed using the so-called affinely equivalent elements; see \cite[Chapter $2.3$]{CiarletBook} and \cite[Chapter $9$]{ErnI} for a classical introduction or \cite{Licht} and \cite{Hiptmair99} for a FEEC-oriented approach.

Fix a reference element $ \That $ and a vector space $\Phat\subset\Lambda^k(\widehat{T})$ of $k$-forms defined on  $ \widehat{T} $. 
For any element $T_i\in\mathcal{T}$, let $ \psi_i: T_i \to \That $ be an invertible affinity such that $ \psi_i (T_i) = \That $. It induces a pullback
$$ \psi_i^* : \quad \Lambda^k (\That) \to \Lambda^k (T_i) $$
which preserves polynomial objects \cite{AFW06,BruniThesis}.
This pullback is declined in various ways depending on $ k $: composition for nodal elements, Piola (covariant and contravariant) mappings for vector fields, and so on. As a consequence, the choice of $\Phat$ induces a space of forms $\mathcal{P}_i$ defined on $T_i$ and basis on $ \That $ induces a basis of $\mathcal{P}_i$.
The FEM space $X$ will be a space of forms defined on $\bigcup_{T_i \in \mathcal{T}}T_i$ such that their restriction to any element $T_i$ belongs to $\mathcal{P}_i$.


\subsection{Bases in duality}\label{sec:bases-duality}
The basis on $\Phat$ will be specified by duality with the preferred degrees of freedom. Let $\{\dof_i:\Phat\to\R\}_{i=1}^{\Nloc}$ be a collection of linear functionals that form a basis of the dual space of $\Phat$, where $\dim\Phat=\Nloc$. 
By duality the degrees of freedom induce a basis $ \{ \omega_i \}_{i=1}^{\Nloc} $
on $\widehat{\mathcal{P}}$, such that
$$ \mathrm{dof}_i (\widehat\omega_j) = \delta_{i,j}.$$

When mappings $  \mathrm{dof}_i $ are \emph{moments} \cite{AFW06,Nedelec80,Nedelec86} or \emph{weights} \cite{m2an,RB09}, these bases are not explicitly known. Nevertheless, if $ \{\widehat\varphi_i \}_{i=1}^{\Nloc} $ is any \emph{convenient} basis, say one easily computable, one may consider the \emph{generalized Vandermonde matrix} \cite{ErnI} $V$, defined as
\begin{equation} \label{eq:VdMdofs}
  V_{i,j} \doteq \mathrm{dof}_i (\widehat\varphi_j),
\end{equation}
in order to obtain an explicit expression for the dual basis \cite{BruniThesis}. In fact, since $ \{\widehat\omega_i \}_{i=1}^{\Nloc} $ and $ \{\widehat\varphi_i \}_{i=1}^{\Nloc} $ are bases for the same vector space $ \Phat $, they must be related as
\begin{equation}\label{eq:W:def}
  \widehat\omega_i = \sum_{j=1}^{\Nloc} W_{i,j} \widehat\varphi_i,
\end{equation}
where matrix $W$ is invertible.
Then
$$ V_{i,j}
= \mathrm{dof}_i (\widehat\varphi_j)
= \mathrm{dof}_{i} \left(\sum_{k=1}^N W_{j,k}^{-1} \widehat\omega_k \right)
= \sum_{k=1}^N W_{j,k}^{-1} \mathrm{dof}_i (\widehat\omega_k)
= W^{-1}_{j,i}
= W^{-T}_{i,j} $$
and, as a consequence, \
\begin{equation}\label{eq:W:VT}
    W = V^{-T}.
\end{equation}

 \subsection{Weights} \label{ssec:weights}

  For the assembly process the above finite elements come with specific degrees of freedom. However, one may keep shape functions and design different and more problem-oriented linear functionals. In this spirit, one may observe that a natural operation on differential forms is \emph{integration}: a differential $ k$-form can be integrated over a manifold of dimension $ k$. The latter suggests the introduction of \emph{weights} \cite{BossavitBook,m2an,RB09} as a generalization of \emph{Lagrangian finite elements}.

  \begin{definition}
    Let $ \omega $ be a smooth $k$-form and let $ s $ be a $k$-simplex. The weight of $ \omega $ over $ s $ is
    \begin{equation}
      w(\omega, s) \doteq \int_s \omega .
    \end{equation}
  \end{definition}

  Weights provide meaningful degrees of freedom for any finite dimensional space of polynomial differential forms 
  $\Phat$ (and therefore for any $ \mathcal{P}_{k} $) as long as the generalized Vandermonde matrix \eqref{eq:VdMdofs} is square and invertible. Notice that, in this case, the $(i,j) $-th element of this matrix is simply
  \begin{equation} \label{eq:VdMweights}
    V_{i,j} = w(\omega_j, s_i) = \int_{s_i} \omega_j ,
  \end{equation}
  being $ \{ s_i \}_{i=1}^{\Nloc} $ a collection of $ \Nloc = \dim \Phat $ $k$-simplices supported in $ \That $. Any set $ \{ s_i \}_{i=1}^{N_{loc}} $ such that $ \det V_{i,j} \ne 0 $ is said to be \emph{unisolvent} for the space $ \Phat $. 
  The search for unisolvent sets is, in general, a tough problem, which has only partial theoretical solutions in the case of simplicial elements \cite{BruniThesis,BBZ22}. A recent work \cite{BE23} showed some explicit examples and obstructions, already in the onedimensional setting.

\subsection{Bilinear form and stiffness matrix}
\label{sec:bilinear:stiffness}

Let $ \Omega \subset \mathbb{R}^d $ and let $ \Lambda^k (\Omega) $ be the space of differential $k$-forms on $ \mathbb{R}^d $. 
We define the differential problem 
\begin{equation} \label{eq:strongHL}
  \begin{cases}
    - L \omega = f, \\
    \text{boundary conditions on } \omega \bigl|_{\partial I} .
  \end{cases}
\end{equation}

Given two smooth functions  $ b,c: \Omega \to \mathbb{R} $, we define $ L: \Lambda^k (\Omega) \to \Lambda^k (\Omega) $ as
\[
  L = b \mathrm{d} c \mathrm{d}^* + c \mathrm{d}^* b \mathrm{d} .
\]
If $ b = c=1 $ this operator is the usual Hodge Laplacian \cite{Warner}
$$
\Delta_H: \quad \omega \mapsto \left( \mathrm{d} \mathrm{d}^* + \mathrm{d}^* \mathrm{d} \right) \omega, $$
%
whose role in PDEs is largely described in \cite{Schwarz}. 
A weak formulation for \eqref{eq:strongHL} is obtained by introducing the $L^2$-inner product on differential forms
\begin{equation} \label{eq:l2innerforms}
  (\omega, \eta)_2 \doteq \int_{\Omega} \omega \wedge \star \eta ,
\end{equation}
see \cite{Lee}. The operator $ \star $ is the Hodge star isomorphism. 
Stokes' Theorem
$$ \int_{\Omega} \de \omega = \int_{\partial \Omega} \omega ,$$
and Leibniz' rule lead to the following integration by parts formula
\begin{equation} \label{eq:intbypartforms}
  \int_{\Omega} \de \omega \wedge \eta = (-1)^{k-1} \int_\Omega \omega \wedge \de \eta + \int_{\partial \Omega} \omega \bigl|_{\partial \Omega} \wedge \eta \bigl|_{\partial \Omega} ,
\end{equation}
which prescribes the boundary conditions to turn the first equation of \eqref{eq:strongHL} into the \emph{Poisson problem for differential forms}. For our purposes, we consider $ d = 1 $ and fix $ c = 1 $; imposing $ \omega \bigl|_{\partial \Omega} = 0 $ and exploiting \eqref{eq:intbypartforms} we deduce the weak formulation
$$ (b \de \de^* \omega, \eta)_2 + (\de^* b \de \omega, \eta)_2 = (f, \eta)_2 \quad \forall \eta \in \Lambda^k (\Omega). $$

%
Either the first or the second term of the above left hand side vanish. 
For $ k = 1 $, since weights reads as integrals over segments supported in $ I = \Omega $, one retrieves discontinuous elements. We thus stick with the case $ k = 0 $, where $ (b \de \de^* \omega, \eta)_2 = 0 $, and hence obtain the bilinear form
$$ a(\omega, \eta) \doteq \int_{I} b \de \omega \wedge \star \de \eta .$$
Identifying $0$-forms with functions, it is immediate to acknowledge that this is the counterpart of the elliptic problem with diffusion coefficient $b(x)$:
$$ \int_I b(x) u'(x) v'(x) = \int_I f(x) v(x) , \quad u \bigl|_{\partial I} = 0 .$$

A FEM Galerkin approach for the approximation of the previous equation leads to
 the linear system
$$ A \widetilde{u} = F ,$$
with $ A $ being the \emph{stiffness} matrix
\begin{equation} \label{eq:stiffness} 
A_{i,j} = a(\varphi_i,\varphi_j) 
= \int_{I} b(x) \varphi_i'(x) \varphi_j'(x) \de x
\end{equation}
and $ F $ the column vector with entries
$$ F_{i} = \int_{I} f(x) \varphi_i (x) \de x, $$
where $\{\varphi_i\}_{i=1}^{N}$ is a basis for the shape functions space $X$. 
Clearly the solution $ \widetilde{u} $ does not depend on the choice of basis functions $ \varphi_i $'s, but it is now explicit that the features of the matrix $ A $ do.

Our analysis will exploit the assembly procedure from local matrices, which is as follows.
Consider a basis for $X$ which is dual to a collection of $N=\dim(X)$ dof's together with a ``local to global'' mapping
\[
\LTG:\quad\{1,\ldots,\Nloc\}\times\mathcal{T}\to \{1,\ldots,N\}
\]
and the assumption that the global $\dof_{\LTG(i,k)}$ will coincide with the local $\dof_i$ when applied to the pullback of forms in $\mathcal{P}_k$; therefore that the restriction of the global basis functions will satisfy
\[
\forall i\in\{1,\ldots,\Nloc\} \,\,
\forall T_k\in\mathcal{T}:
\quad
\left. \varphi_{\LTG(i,k)} \right\vert_{T_k}
 =  \phi_k^* \left(\widehat\omega_i\right) .
\]
Then, for any element $T_k\in\mathcal{T}$, one introduces the local matrices by 
restricting the integrals in the weak formulation 
to one specific element $T_k$ and changing variables to perform integration on the reference element as
\begin{equation}\label{eq:Aloc}
A_{{i},{j}}^{(T_k)}
= \int_{\That}
       b(\psi_k(\hat{x}))
       \widehat{\omega}_{{i}}' (\hat{x})
       \widehat{\omega}_{{j}}' (\hat{x})
       \de \hat{x}
       .
\end{equation}
The global stiffness matrix can then be assembled as
\[
A=
  \sum_{T_k\in\mathcal{T}}
  \sum_{{i}=1}^{\Nloc}
  \sum_{{j}=1}^{\Nloc}
  \mathrm{E}_{\LTG({i},k)\,,\,\LTG({j},k)} \left( A^{(T_k)}_{{i},{j}}\right)
\]
where $\mathrm{E}_{i,j}(x)$ is an $N\times N$ matrix with only one nonzero entry, equal to $x$, at location $(i,j)$.

This yields the following description of the matrix of the system.

  \begin{proposition} \label{prop:bilinearchange}
  Let $A$ and $B$ be the stiffness matrices written with respect to the bases 
  $ \{\varphi_i\}_{i=1}^{N} $ and respectively
  $ \{\omega_i\}_{i=1}^{N} $;  
  there exists an invertible matrix $ M $ such that
  \begin{equation} \label{eq:changepart1}
      A = M^{-T} B M^{-1} .
    \end{equation}
    Moreover, the elemental matrices $ A^{(T_k)} $ and $B^{(T_k)} $ for $T_k\in\mathcal{T}$
    written with respect to  $ \{\widehat\varphi_i\}_{i=1}^{\Nloc} $
    and respectively to $ \{\widehat\omega_i\}_{i=1}^{\Nloc} $, satisfy
    \begin{equation} \label{eq:changepart2}
      A^{(T_k)} = V^{-T} B^{(T_k)} V^{-1} .
    \end{equation}
  \end{proposition}

  \begin{proof}
    Equality \eqref{eq:changepart1} is a general fact of linear algebra. 
    Equality \eqref{eq:changepart2} is a direct consequence of \eqref{eq:Aloc} and \eqref{eq:W:VT}, being $ \{\widehat\varphi_i\}_{i=1}^{N_{loc}} $ the basis in duality with the degrees of freedom.
  \end{proof}

  \begin{remark}
    Equation \eqref{eq:changepart2} mixes ingredients of very different nature: matrices $ A^{(T_k)} $ and $ B^{(T_k)} $ are local and depend on the bilinear form $ a $, see Eq. \eqref{eq:Aloc}. In contrast, the generalized Vandermonde matrix does not depend on $ a $, see Eq. \eqref{eq:VdMweights}, and is \emph{global} in the following sense.
    Let $V^{(T_k)}$ be the generalized Vandermonde matrix relating the two local bases
    $\{\psi^*(\widehat{\varphi_i})\}$ and $\{\psi^*(\widehat{\omega_i})\}$ on the element $T_k$. With reference to Eq. \eqref{eq:VdMweights}, one has
    $$
    V^{(T_k)}_{i,j} 
    = \int_{\psi_k^{-1}(s_i)} \psi_k^*(\widehat\omega_j)
    = \int_{s_i} \widehat\omega_j
    = V_{i,j}.
    $$
  \end{remark}

  Notice that, in general, the matrix $ M $ in \eqref{eq:changepart1} may be tough to be computed. In contrast, a recipe for $ V $ is neat; see Eq. \eqref{eq:VdMdofs}. 
  Once a set of local $\dof$'s is fixed, when the corresponding basis functions cannot be explicitly computed, relationship \eqref{eq:changepart2} can be used to compute the local matrices during assembly.
  Notice however that too naive Vandermonde matrices, such that associated with monomial basis and nodes on the real line, are inherently exponentially ill-conditioned as a function of the matrix order
  \cite{S-bumi}. It is plain to observe that this aspect is not secondary and should be taken into consideration.

\section{Spectral tools and useful matrix structures}\label{sec:spectral-tools}

When the number $n$ of elements in the mesh increases, 
the size of the linear system will also increase, together with its conditioning. 
We are concerned with the spectral analysis of the matrix structures and the design of appropriate preconditioners for Krylov methods, with the idea of obtaining robustness and optimality with respect to the relevant parameters.
Our tools come from block Toeplitz operators and block GLT matrix-sequences.
To this end, in Section \ref{ssez:matrix-seq} we introduce the notion of distribution in the Weyl sense and clustering, when a general matrix-sequence is considered. Section \ref{ssez:Toeplitz-etc} is devoted to block Toeplitz, circulant, sampling diagonal matrices, while Section \ref{ssez:GLT} treats the definition of the $*$-algebra of GLT sequences, with a special emphasis on the operative features.
\subsection{Clustering and distribution}\label{ssez:matrix-seq}
We begin with the notion of distribution in the Weyl sense, when a
matrix-valued symbol occurs, both in the sense of the eigenvalues and
singular values. The notion of clustering at a given fixed point, both in the sense of
eigenvalues and singular values, can be seen as a special case of the distribution notions.
\par \noindent
$\bullet$ \textbf{Sequences of Matrices and Block
  Matrix-Sequences.} A sequence of matrices is any sequence of the form $\{A_n\}_n$, where $A_n$ is a square
matrix of size $d_n$, $d_n$ is monotonically strictly increasing so that $d_n\to\infty$ as $n\to\infty$.
Let $r\ge1$ be a fixed positive integer independent of $n$. A $r$-block matrix-sequence - or simply a matrix-sequence if $r$ can
be inferred from the context or we do not need/want to specify it - is a special sequence of matrices $\{A_n\}_n$ in which the size of
$A_n$ is $d_n=r\phi_n$,  with $\{\phi_n\}_{n}$ being a strictly increasing sequence of positive integers.
In the following sections,
$A_n$ will be the stiffness matrix \eqref{eq:stiffness} obtained when partitioning the domain into $n$  elements.

\par \noindent
$\bullet$ \textbf{Singular Value and Eigenvalue Distribution of a
  Sequence of Matrices.} Let $\mu_t$ be the Lebesgue measure in
$\mathbb R^t$, $t\ge 1$. Throughout the current work, all the
terminology from measure theory (such as ``measurable set'',
``measurable function'', ``a.e.'', etc.) is referred to the
Lebesgue measure. A matrix-valued function $f:D\subseteq\mathbb
R^t\to\mathbb{C}^{r\times r}$ is said to be measurable (resp.,
continuous, Riemann-integrable, in $L^p(D)$, etc.) if its
components $f_{\alpha\beta}:D\to\mathbb{C},\
\alpha,\beta=1,\ldots,r$, are measurable (resp., continuous,
Riemann-integrable, in $L^p(D)$, etc.). We denote by
${C}_c(\mathbb R)$ (resp., ${ C}_c(\mathbb C)$) the space of
continuous complex-valued functions with bounded support defined
on $\mathbb R$ (resp., $\mathbb{C}$). If $A\in\mathbb{C}^{m\times
  m}$, the singular values and the eigenvalues of $A$ are denoted by
$\sigma_1(A),\ldots,\sigma_m(A)$ and
$\lambda_1(A),\ldots,\lambda_m(A)$, respectively.
\begin{definition}\label{def-distribution}
  Let $\{A_n\}_n$ be a sequence of matrices, with $A_n$ of size
  $d_n$, and let $f:D\subset\mathbb R^t\to\mathbb{C}^{r\times r}$ be
  a measurable function defined on a set $D$ with
  $0<\mu_t(D)<\infty$.
  \begin{itemize}
    \item We say that $\{A_n\}_n$ has a (asymptotic) singular value distribution described by $f$, and we write $\{A_n\}_n\sim_\sigma f$, if
    \begin{equation}\label{distribution:sv-sv}
      \lim_{n\to\infty}\frac1{d_n}\sum_{i=1}^{d_n}F(\sigma_i(A_n))=\frac1{\mu_t(D)}\int_D\frac{\sum_{i=1}^{r}F(\sigma_i(f(\mathbf x)))}{r}{\rm d}\mathbf x,\qquad\forall\,F\in C_c(\mathbb R).
    \end{equation}
    \item We say that $\{A_n\}_n$ has a (asymptotic) spectral (or eigenvalue) distribution described by $f$, and we write $\{A_n\}_n\sim_\lambda f$, if
    \begin{equation}\label{distribution:sv-eig}
      \lim_{n\to\infty}\frac1{d_n}\sum_{i=1}^{d_n}F(\lambda_i(A_n))=\frac1{\mu_t(D)}\int_D\frac{\sum_{i=1}^{r}F(\lambda_i(f(\mathbf x)))}{r}{\rm d}\mathbf x,\qquad\forall\,F\in C_c(\mathbb C).
    \end{equation}
  \end{itemize}
  If $\{A_n\}_n$ has both a singular value and an eigenvalue distribution described by $f$, we write $\{A_n\}_n\sim_{\sigma,\lambda}f$.
\end{definition}
Definition~\ref{def-distribution} is well-posed owing to the Lebesgue measurability of the functions
$\mathbf x\mapsto\sum_{i=1}^{r}F(\sigma_i(f(\mathbf x)))$ and $\mathbf x\mapsto\sum_{i=1}^{r}F(\lambda_i(f(\mathbf x)))$.
Whenever we write a relation such as $\{A_n\}_n\sim_\sigma f$ or
$\{A_n\}_n\sim_\lambda f$, it is understood that $f$ is as in
Definition~\ref{def-distribution}; that is, $f$ is a measurable
function defined on a subset $D$ of some $\mathbb R^t$ with
$0<\mu_t(D)<\infty$ and taking values in $\mathbb C^{r\times r}$
for some $r\ge1$.   The informal meaning behind the spectral
distribution \eqref{distribution:sv-eig} is the following: if $f$
is continuous, then a suitable ordering of the eigenvalues
$\{\lambda_j(A_n)\}_{j=1,\ldots,d_n}$, assigned in correspondence
with an equispaced grid on $D$, reconstructs approximately the $r$
surfaces  $\mathbf x$ $\mapsto\lambda_i(f(\mathbf x)),\
i=1,\ldots,r$. For instance, if $t=1$, $d_n=nr$, and $D=[a,b]$,
then the eigenvalues of $A_n$ are approximately equal to
$\lambda_i(f(a+j(b-a)/n))$, $j=1,\ldots,n,\ i=1,\ldots,r$; if
$t=2$, $d_n=n^2 r$, and $D=[a_1,b_1]\times [a_2,b_2]$, then the
eigenvalues of $A_n$ are approximately equal to
$\lambda_i(f(a_1+j_1(b_1-a_1)/n,a_2+j_2(b_2-a_2)/n)),\
j_1,j_2=1,\ldots,n,\ i=1,\ldots,r$ (and so on for $t\ge3$). This
type of information is useful in engineering applications
\cite{tom-paper}, e.g. for the computation of the relevant
vibrations, and in the analysis of the (asymptotic) convergence
speed of iterative solvers for large linear systems or for
improving the convergence rate by e.g. the design of appropriate
preconditioners \cite{BK,BeSe}.
\par
The next theorem gives useful tools for computing the spectral
distribution of sequences formed by Hermitian matrices. For the
related proof, we refer the reader to
\cite[Theorem~4.3]{opiccolo}. In what follows, the conjugate
transpose of the matrix $A$ is denoted by $A^*$. If $A\in\mathbb
C^{m\times m}$ and $1\le p\le\infty$, we denote by $\|A\|_p$ the
Schatten $p$-norm of $A$, i.e., the $p$-norm of the
vector $(\sigma_1(A),\ldots,\sigma_m(A))$. 
The Schatten $\infty$-norm $\|A\|_\infty$ is the largest singular value of $A$ and coincides
with the spectral norm $\|A\|$. The Schatten 1-norm $\|A\|_1$ is the sum of the singular values of $A$ and is
often referred to as the trace-norm of $A$. The Schatten 2-norm $\|A\|_2$ coincides with the Frobenius norm of $A$. Schatten $p$-norms are treated in detail in a beautiful book by Bhatia \cite{B}.
\begin{theorem}\label{extradimensional}
  Let $\{X_n\}_n$ be a sequence of matrices, with $X_n$ Hermitian of size $d_n$, and let $\{P_n\}_n$
  be a sequence such that $P_n\in\mathbb C^{d_n\times\delta_n}$, $P_n^*P_n=I_{\delta_n}$, $\delta_n\le d_n$
  and $\delta_n/d_n\to1$ as $n\to\infty$. Then, $\{X_n\}_n\sim_{\sigma,\lambda} f$ if and only if $\{P_n^*X_nP_n\}_n\sim_{\sigma,\lambda} f$.
\end{theorem}
\par
Now we turn to the definition of clustering. For $z\in\mathbb C$ and $\epsilon>0$, let $B(z,\epsilon)$
the disk with center $z$ and radius $\epsilon$, $B(z,\epsilon)\doteq\{w\in\mathbb C:\,|w-z|<\epsilon\}$.
For $S\subseteq\mathbb C$ and $\epsilon>0$, we denote by $B(S,\epsilon)$ the $\epsilon$-expansion of $S$,
defined as $B(S,\epsilon)\doteq\bigcup_{z\in S}B(z,\epsilon)$.
\begin{definition}\label{def-cluster}
  Let $\{X_n\}_n$ be a sequence of matrices, with $X_n$ of size $d_n$ tending to infinity, and
  let $S\subseteq\mathbb C$ be a nonempty closed subset of $\mathbb C$. $\{X_n\}_n$ is {\em strongly clustered} at $S$
  in the sense of the eigenvalues if, for each $\epsilon>0$, the number of eigenvalues of $X_n$ outside $B(S,\epsilon)$
  is bounded by a constant $q_\epsilon$ independent of $n$. In symbols,
  $$q_\epsilon(n,S)\doteq\#\{j\in\{1,\ldots,d_n\}: \lambda_j(X_n)\notin B(S,\epsilon)\}=O(1),\quad\mbox{as $n\to\infty$.}$$
  $\{X_n\}_n$ is {\em weakly clustered} at $S$ if, for each $\epsilon>0$,
  $$q_\epsilon(n,S)=o(d_n), \quad \mbox{as $n\to\infty.$}$$
  If $\{X_n\}_n$ is strongly or weakly clustered at $S$ and $S$ is not connected, then the connected components of $S$ are called sub-clusters.
\end{definition}
\begin{remark}\label{rem:clustering vs distribution}
  Recall that, for a measurable function $g:D\subseteq\mathbb
  R^t\to\mathbb C$, the essential range of $g$ is defined as
  $\mathcal{ER}(g)\doteq\{z\in\mathbb C:\,\mu_t(\{g\in
  B(z,\epsilon)\})>0\mbox{ for all $\epsilon>0$}\}$, where $\{g\in
  B(z,\epsilon)\}\doteq\{x\in D:\,g(x)\in B(z,\epsilon)\}$.
  $\mathcal{ER}(g)$ is always closed; moreover, if $g$ is continuous
  and $D$ is contained in the closure of its interior, then
  $\mathcal{ER}(g)$ coincides with the closure of the image of $g$.
  \par
  Hence, if $\{X_n\}_n\sim_\lambda f$ (with $\{X_n\}_n,\  f$ as in Definition \ref{def-distribution}), then,
  by \cite[Theorem 4.2]{gol-serra}, $\{X_n\}_n$ is weakly clustered at the essential range of $f$, defined
  as the union of the essential ranges of the eigenvalue functions $\lambda_i(f),\ i=1,\ldots, r$: $\mathcal{ER}(f)\doteq\bigcup_{i=1}^r\mathcal{ER}(\lambda_i(f))$. Finally, if $\mathcal{ER}(f)=s$ with $s$ fixed complex number and $\{X_n\}_n\sim_\lambda f$, then $\{X_n\}_n$ is weakly clustered at $s$.
  All the considerations above can be translated in the singular value setting as well, with obvious minimal modifications.
 For further relationships among distribution, clustering, Schatten $p$-norms the reader is referred to \cite{taud2} (see also \cite{ty-1}).
\end{remark}
\subsection{Block Toeplitz/Circulant/diagonal matrices and zero-distributed matrix-sequences}\label{ssez:Toeplitz-etc}
In this subsection we introduce three types of matrix  structures. The first two have an algebraic definition for every fixed
dimension (block Toeplitz/Circulant matrices and block diagonal structures), while the last has only an asymptotic sense (zero-distributed
matrix-sequences). As it will be clear in the sequel, starting from Subsection \ref{ssez:GLT}, we deal with matrix-sequences  consisting of these three matrix structures, especially when defining the basics of the theory of block GLT sequences.
\par \noindent
$\bullet$ \textbf{Block Toeplitz/Circulant Matrices.}
We concisely summarize the definition and relevant properties of block Toeplitz matrices.\\
Given $r,n$ fixed positive integers i.e. given $r,{n}\in \mathbb{N}^+$, a matrix of the form
\[
[A_{{i}-{j}}]_{{i},{j}=1}^{{n}} \in \mathbb{C}^{{n}r \times n r}
\]
with blocks $A_{k} \in \mathbb{C}^{r\times r}$, $k =-({n}-1), \ldots, {n}-1$, is called block Toeplitz matrix, or, more precisely, $r$-block Toeplitz matrix. Given a matrix-valued function $f : [-\pi, \pi] \rightarrow \mathbb{C}^{r\times r}$ belonging to $L^1([-\pi, \pi])$, we denote its Fourier coefficients by
\begin{equation}\label{eq:fourier_coefficients}
  \hat{f}_{k} =\frac{1}{2\pi} \int_{[-\pi,\pi]} f({\theta})e^{-\hat i\, k\theta}\, d\theta \in \mathbb{C}^{r\times r}, \ \ \ {k}
  \in \mathbb{Z}.
\end{equation}
For every ${n} \in \mathbb{N}^+$, the ${n}$-th Toeplitz matrix associated with $f$ is defined as
\begin{equation}\label{eq:toeplitz}
  T_{n}(f) \doteq [\hat{f}_{{i}-{j}}]_{{i},{j}=1}^{{n}}
\end{equation}
or, equivalently, as
\begin{equation}\label{eq:toeplitz_kron}
  T_{n}(f) = \sum_{|j|<n}  J_{n}^{(j)} \otimes \hat{f}_{j}
\end{equation}
where $\otimes$ denotes the (Kronecker) tensor product of matrices, while $J_m^{(l)}$ is the matrix of order $m$ whose $(i,j)$
entry equals $1$ if $i-j=l$ and zero otherwise.
We call $\{T_{n}(f)\}_{{n}\in \mathbb{N}^+}$ the family of (block) Toeplitz matrices associated with $f$, which, in turn, is called the generating function of $\{T_{n}(f)\}_{{n}\in \mathbb{N}^+}$.
In perfect analogy we define block Circulant matrices. Given $r,{n}\in \mathbb{N}^+$, a matrix of the form
\[
[A_{({i}-{j})\ {\rm mod\ }{n} }]_{{i},{j}=1}^{{n}} \in \mathbb{C}^{{n}r \times {n}r}
\]
with blocks $A_{k} \in \mathbb{C}^{r\times r}$, ${k} = {0}, \ldots, {n}-1$, is called block Circulant matrix, or, more precisely, $r$-block Circulant matrix.
\par
The ${n}$-th Circulant matrix associated with $f$ is defined as
\begin{equation}\label{eq:circ}
  C_{n}(f) \doteq \sum_{|j|<n}  Z_{n}^{(j)} \otimes \hat{f}_{j}
\end{equation}
where  $Z_m^{(l)}=[Z_m^{(1)}]^l$ is the matrix of order $m$ whose $(i,j)$ entry equals $1$ if $(i-j) \ {\rm mod\ } m=l$ and zero otherwise.
\par \noindent
$\bullet$ \textbf{Block Diagonal Sampling Matrices.} For
$n\in\mathbb N^+$, $d=1$, and $a:[0,1]\to\mathbb C^{r\times r}$, we
define the block diagonal sampling matrix $D_n(a)$ as the block diagonal
matrix
\[
D_n(a) \doteq \mathop{\rm diag}_{i=1,\ldots,n}a\Bigl(\frac{i}{n}\Bigr)=\left[\begin{array}{cccc}a(\frac1n) & & & \\ & a(\frac2n) & & \\ & & \ddots & \\
  & & & a(1)\end{array}\right]\in\mathbb C^{rn\times rn}.
\]
\par \noindent $ \bullet $ \textbf{Zero-Distributed Sequences.}
A sequence of matrices $\{Z_n\}_n$ such that $\{Z_n\}_n\sim_\sigma0$ is referred to as a zero-distributed sequence.
Note that, for any $r\ge1$, $\{Z_n\}_n\sim_\sigma0$ is equivalent to $\{Z_n\}_n\sim_\sigma O_r$ (throughout this paper, $O_m$ and $I_m$
denote the $m\times m$ zero matrix and the $m\times m$ identity matrix, respectively). Theorem~\ref{0cs} provides an important
characterization of zero-distributed sequences together with a useful sufficient condition for detecting such sequences \cite{taud2}. Throughout this paper we use the natural convention $1/\infty=0$.
\begin{theorem}\label{0cs}
  Let $\{Z_n\}_n$ be a sequence of matrices, with $Z_n$ of size $d_n$, and let $\|\cdot\|$ be the spectral norm. Then
  \begin{itemize}
    \item $\{Z_n\}_n$ is zero-distributed if and only if $Z_n=R_n+N_n$ with ${\rm rank}(R_n)/d_n\to0$ and $\|N_n\|\to0$ as $n\to\infty$.
    \item $\{Z_n\}_n$ is zero-distributed if there exists a $p\in[1,\infty]$ such that $\|Z_n\|_p/(d_n)^{1/p}\to0$ as $n\to\infty$.
  \end{itemize}
\end{theorem}
\subsection{The $*$-algebra of block GLT sequences}\label{ssez:GLT}
Let $r\ge1$ be a fixed positive integer. A $r$-block GLT sequence (or simply a GLT sequence if $r$ can be inferred from the context
or we do not need/want to specify it) is a special $r$-block matrix-sequence $\{A_n\}_n$ equipped with a measurable function
$\kappa:[0,1]\times[-\pi,\pi]\to\mathbb C^{r\times r}$, the so-called symbol or GLT symbol. We use the notation $\{A_n\}_n\sim_{\rm
  GLT}\kappa$ to indicate that $\{A_n\}_n$ is a GLT sequence with symbol $\kappa$. The symbol of a GLT sequence is unique in the
sense that if $\{A_n\}_n\sim_{\rm GLT}\kappa$ and $\{A_n\}_n\sim_{\rm GLT}\varsigma$ then $\kappa=\varsigma$ a.e. in the definition domain $[0,1]\times[-\pi,\pi]$.

The main properties of $r$-block GLT sequences proved in \cite{etnaGLTbookIII} are listed below: they represent a complete characterization of GLT sequences, equivalent to the full constructive definition given originally in \cite{glt-1} for $r=1$ (see also the seminal paper by Tilli \cite{Tilli-loc} for a previous less general construction).

If $A$ is a matrix, we denote by $A^\dag$ the Moore--Penrose pseudoinverse of $A$ (recall that $A^\dag=A^{-1}$
whenever $A$ is invertible). If $f_m,f:D\subseteq\mathbb R^t\to\mathbb C^{r\times r}$ are measurable matrix-valued
functions, we say that $f_m$ converges to $f$ in measure (resp., a.e., in $L^p(D)$, etc.) if $(f_m)_{\alpha\beta}$ converges to
$f_{\alpha\beta}$ in measure (resp., a.e., in $L^p(D)$, etc.) for all $\alpha,\beta=1,\ldots,r$.
\begin{enumerate}
  \item[\textbf{GLT\,1.}] If $\{A_n\}_n\sim_{\rm GLT}\kappa$ then $\{A_n\}_n\sim_\sigma\kappa$ in the sense of Definition \ref{def-distribution} with $t=2$. If moreover each $A_n$ is Hermitian then $\{A_n\}_n\sim_\lambda\kappa$, again in the sense of Definition \ref{def-distribution} with $t=2$.
  \item[\textbf{GLT\,2.}] We have:
  \begin{itemize}
    \item $\{T_{n}(f)\}_{n}\sim_{\rm GLT}\kappa({x},{\theta})=f({\theta})$ if $f:[-\pi,\pi]\to\mathbb C^{r\times r}$ is in $L^1([-\pi,\pi])$;
    \item $\{D_{n}(a)\}_{n}\sim_{\rm GLT}\kappa({x},{\theta})=a({x})$ if $a:[0,1]\to\mathbb C^{r\times r}$ is Riemann-integrable;
    \item $\{Z_n\}_n\sim_{\rm GLT}\kappa({x},{\theta})=O_r$ if and only if $\{Z_n\}_n\sim_\sigma 0$.
  \end{itemize}
  \item[\textbf{GLT\,3.}] If $\{A_n\}_n\sim_{\rm GLT}\kappa$ and $\{B_n\}_n\sim_{\rm GLT}\varsigma$ then:
  \begin{itemize}
    \item $\{A_n^*\}_n\sim_{\rm GLT}\kappa^*$;
    \item $\{\alpha A_n+\beta B_n\}_n\sim_{\rm GLT}\alpha\kappa+\beta\varsigma$ for all $\alpha,\beta\in\mathbb C$;
    \item $\{A_nB_n\}_n\sim_{\rm GLT}\kappa\varsigma$;
    \item $\{A_n^\dag\}_n\sim_{\rm GLT}\kappa^{-1}$ provided that $\kappa$ is invertible a.e.
  \end{itemize}
  \item[\textbf{GLT\,4.}] $\{A_n\}_n\sim_{\rm GLT}\kappa$ if and only if there exist $r$-block GLT sequences $\{B_{n,m}\}_n\sim_{\rm GLT}\kappa_m$ such that $\{B_{n,m}\}_n\stackrel{\rm a.c.s.}{\longrightarrow}\{A_n\}_n$ and $\kappa_m\to\kappa$ in measure, where the a.c.s. convergence is studied in \cite{GLTbookI}.
\end{enumerate}

\section{Spectral symbol and spectral theory} \label{sec:spectral}
    The present section is devoted to the spectral and structure analysis for fixed dimension and for the related sequences of matrices as the fineness parameter tends to zero.
    To begin with, 
    we explicitly construct the symbol, following the assembly procedure of the stiffness matrix, which is depicted in Figure~\ref{fig:Ablocks}. 
    In our setting weigths dof's are just evaluation at points, so let $\{\xi_i\}_{i=1}^{r+1}$ be a set of reference positions in $\Ihat=[0,1]$.
    We recall that the presence of both extrema of the interval is needed for defining continuous elements, so we assume that $\xi_1=0$ and $\xi_{r+1}=1$.

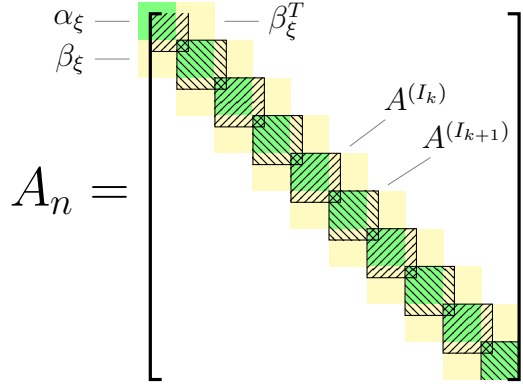
\begin{figure}[h]
\begin{center}
\begin{tikzpicture}[scale=0.5]
  [pin distance=10mm,
   pin edge={stealth-,shorten <=-4pt,decorate,decoration=bent}]
  \foreach \x in {0,1,2,3,4,5,6,7,8,9}{
    \fill[green!50]  (\x,-\x) rectangle +(1,-1);
  }
  \foreach \x in {0,1,2,3,4,5,6,7,8}{
    \fill[yellow!30] (\x+1,-\x) rectangle +(1,-1);
    \fill[yellow!30] (\x,-\x-1) rectangle +(1,-1);
  }

  \begin{scope}
    \clip (0.3,-0.3) rectangle (10,-10);
    \foreach \x in {0,2,4,6,8}{
     \filldraw[draw=black,pattern=north east lines]
         (\x,-\x)  rectangle +(1.3,-1.3);
     \filldraw[draw=black,pattern=north west lines]
         (\x+1,-\x-1)  rectangle +(1.3,-1.3);
    }
  \end{scope}
  \node[pin=above right:{\large $A^{(I_k)}$}] at (5.3,-4) {};
  \node[pin=above right:{\large $A^{(I_{k+1})}$}] at (6.3,-5) {};

  \node[pin=left:{\large $\alpha_\xi$}] at (0,-0.5) {};
  \node[pin=left:{\large $\beta_\xi$}] at (0,-1.5) {};
  \node[pin=right:{\large $\beta_\xi^T$}] at (2,-0.5) {};

  \draw[very thick,xshift=0.3cm] (0.3,-0.3) -- (0.,-0.3) -- (0.,-10) -- (0.3,-10);
  \draw[very thick,xshift=10cm]  (-0.3,-0.3) -- (0.,-0.3) -- (0.,-10) -- (-0.3,-10);
  \node[anchor=east] at (0,-5) {\huge $A_n=$};
\end{tikzpicture}
\end{center}
    \caption{Structure of the stiffness matrix and of the Toeplitz matrix.
    Element contributions are hatched in black: note the overlap of nearby blocks and the smaller first and last block. The blocks of the Toeplitx matrix are shaded in colour.}
    \label{fig:Ablocks}
\end{figure}

    Let then $A^{(I_k)}$ be the $(r+1)\times(r+1)$ local contribution of element $I_k$ to the global stiffness matrix, defined as in \eqref{eq:Aloc}. In the simple case of unit diffusion coefficient ($b(x)=1$) and uniform mesh size, then $A^{(I_k)}=A^{\loc}$ is independent of $k$ and the stiffness matrix has the general structure of a block-tridiagonal Toeplitz matrix with $r\times r$ blocks (see Figure~\ref{fig:Ablocks}).
    We are thus able to write the spectral symbol
    $f_{r}^{\xi}: [-\pi, \pi]  \to \mathbb{C}^{r \times r}$
    of the stiffness matrix as a function depending on the $ \xi_i $'s.

    Let us consider a uniform grid with $n$ elements and let $A_n$ be the corresponding stiffness matrix \eqref{eq:stiffness}. 
    Consider the set of global indices $\mathcal{I}_k\doteq\{\LTG(i,k):i=0,\ldots,r\}$ associated with a generic $k$-th element (with $k\neq1,n$ since the first and last elements are influenced by the boundary conditions). 
    Since the contribution of the shape function associated with $\xi_{r+1}=1$ of each element will overlap in $A_n$ with those of of the shape function associated with $\xi_1=0$ of the next element, the submatrix of $A_n$ for the rows and columns in $\mathcal{I}_k$ is
    $$ (A_\xi)_{i,j} \doteq
    \begin{cases}
      A_{1, 1}^{\loc} + A_{r+1, r+1}^{\loc} &\text{ if } (i,j)=(1,1) \text{ or } (r+1,r+1) , \\
      A_{i,j}^{\loc} &\text{ otherwise }
    \end{cases}
    $$
    The structure of the $(r+1)\times (r+1)$ matrix $A_\xi$ is
    $$ A_{\xi} =\left[
    \begin{array}{c|c}
      \alpha_\xi & \widetilde{\beta}_\xi^T \\ \hline
      \widetilde{\beta}_\xi & b
    \end{array}\right],$$
    where $ \alpha_\xi \in \mathbb{R}^{r \times r}$, $ \widetilde{\beta}_\xi \in \mathbb{R}^{1 \times r}$, $ b \in \mathbb{R} $ and
    $b=(\alpha_\xi)_{1,1}$. Following \cite{Rahla}, define 
    \begin{equation*}
      \beta_\xi \doteq \left[ \begin{array}{c}
        \widetilde{\beta}_\xi \\ \hline
        \boldsymbol{0}
      \end{array}
      \right] ,
    \end{equation*}
    with $ \boldsymbol{0} $ being the zero matrix of size $(r-1)\times r$, so that $ \beta_\xi \in \mathbb{R}^{r \times r}$.

    Recalling (\ref{eq:fourier_coefficients})-(\ref{eq:toeplitz_kron}),
    we observe that the stiffness matrix $A_n$ is a principal submatrix of size $nr-1$ of the Toeplitz matrix $T_n(f_{r}^{\xi})$ with generating function
    \begin{align} \label{eq:symboldegr}
      f_{r}^{\xi}: \quad [-\pi, \pi] & \to \mathbb{C}^{r \times r},
      \\
      \theta & \mapsto \alpha_\xi + \beta_\xi e^{i \theta} + \beta_\xi^T e^{-i \theta}. \nonumber
    \end{align}

    Furthermore, according to the first item of axiom {\bf GLT2} and to axiom {\bf GLT1}, we deduce that $f_{r}^{\xi}$ is also the spectral and the singular value symbol of $\{T_{n}(f_{r}^{\xi})\}_{{n}}$, in the sense of Definition \ref{def-distribution}. As a first conclusion we infer that $f_{r}^{\xi}$ is the spectral and the singular value symbol of the stiffness matrix-sequence $\{A_{n}\}_{{n}}$. This is proven by applying Theorem \ref{extradimensional} to $X_n=T_{n}(f_{r}^{\xi})$, $A_n=P_n^*X_nP_n$, $d_n=nr$, $\delta_n=nr-1$, so that $\delta_n/d_n\to1$ as $n\to\infty$, and $P_n$ rectangular matrix obtained by the identity of size $d_n$ by eliminating only the last column.

    As a consequence, the union of the ranges of the $r$ eigenvalues $\lambda_1(f_{r}^{\xi})\le \lambda_2(f_{r}^{\xi}) \le \cdots \le \lambda_r(f_{r}^{\xi})$
    is a weak cluster of the eigenvalues of both $\{T_{n}(f_{r}^{\xi})\}_{{n}}$ and $\{A_{n}\}_{{n}}$, in accordance with Definition \ref{def-cluster} and the subsequent Remark \ref{rem:clustering vs distribution}.

\subsection{Determinant of the symbol}
The aim of this section is to provide an explicit formula for the determinant of the symbol \eqref{eq:symboldegr}.
We start recalling that in the case of the standard Lagrangian elements (i.e. $\xi_i=(i-1)/r$ for $i=1,\ldots,r+1$)
$$ \det f_r(\theta) = d_r (2 - 2 \cos(\theta)) ,$$
with
\begin{equation}\label{eq:dr}
   d_r \doteq \det \left[ < \ell_i', \ell_j' > \right]_{i,j=2}^r ,
\end{equation}
where $ \{\ell_i\}_{i=1}^{r+1} $ is the Lagrangian basis associated with \emph{equispaced} nodes in $[0,1]$.
For the proof, see \cite[Theorem 8]{GSCS}, two ingredients are employed. The first is basic and concerns the properties of the determinant, when expanding its computation along rows and columns. The second ingredient is the relation
$$ \det (f_r (0) ) = d_r + 2d_r' + d_r'' = 0 \quad \text{and} \quad d_r + d_r' = 0, $$
with $ d_r' $ and $ d_r'' $ defined as in \cite[Eq. (45)]{GSCS}, which can be easily proven once the following is established
$$ \sum_{j=1}^{r+1} \ell_i' (x) = 0.$$
In order to extend the result on $\det f_r$ to the case of basis induced by weights and hence to $\det f_r^{\xi}$, we need to generalise the previous result.

\begin{lemma} \label{lem:differentiallagrange}
  Let $ \{ \varphi_i \}_{i=1}^{r+1} $ the basis for $ \mathbb{P}_r $ induced by weights corresponding to evaluations on $ \{ \xi_i \}_{i=1}^{r+1} $. Then
  $$ \sum_{i=1}^{r+1} \varphi_i' (x) = 0 .$$
  Moreover, 
  any choice of $ r $ polynomials out of $ \{ \varphi_i' \}_{i=1}^{r+1}  $ is a basis for $ \mathbb{P}_{r-1} $.
\end{lemma}

\begin{proof}
  We start by defining $ q(x) $ as the degree $ r $ polynomial
  $$ q(x) \doteq \sum_{i=1}^{r+1} \varphi_i (x) - 1. $$
  As the basis $ \{ \varphi_i \}_{i=1}^{r+1} $ is dual to weights associated with $ \{ \xi_i \}_{i=1}^{r+1} $, one has $ q(\xi_i) = 0 $ for each $ \xi_i $, whence $ q(x) = 0 $ and so $\sum_{i=1}^{r+1} \varphi_i (x) = 1$ identically. By differentiating both sides one obtains the first claim. The second part of the proof is identical to that already known for the Lagrange basis (see \cite{GSCS}).
\end{proof}

Equation~\eqref{eq:changepart2} motivates the following computation.


\begin{proposition} \label{prop:computingVanderm}
  Let $ V_{i,j} = \ell_j (\xi_i) $ be the Vandermonde matrix written with respect to the uniform Lagrangian basis $ \{\ell_i \}_{i=1}^{r+1} $ and some general collection of points $ \{ \xi_i \}_{i=1}^{r+1} $ such that $\xi_1=0$ and $\xi_{r+1}=1$. Then
  \begin{equation} \label{eq:inverseVdM}
    V^{-1} = \left[
    \begin{array}{c|c|c}
      1 & \boldsymbol{0} & 0\\ \hline
      -\boldsymbol{x}_V & (V')^{-1} & -\boldsymbol{y}_V \\ \hline
      0 & \boldsymbol{0} & 1
    \end{array}\right] ,
  \end{equation}
  where $ V' \doteq [V_{i,j}]_{i,j=2}^{r} $ and $ \boldsymbol{x}_V \doteq V' \boldsymbol{x} $ and $ \boldsymbol{y}_V \doteq V' \boldsymbol{y} $.
\end{proposition}
\begin{proof}
  Since $\xi_1=0$ and $\xi_{r+1}=1$, $ V_{1,j} = \ell_j (\xi_1) = \delta_{j,1} $ and $ V_{r+1,j} = \ell_j (\xi_{r+1}) = \delta_{j,r+1} $. Thus
  \begin{equation*}
    V =\left[
    \begin{array}{c|c|c}
      1 & \boldsymbol{0} & 0\\ \hline
      \boldsymbol{x} & V' & \boldsymbol{y} \\ \hline
      0 & \boldsymbol{0} & 1
    \end{array}\right] = \left[
    \begin{array}{c|c|c}
      1 & \boldsymbol{0} & 0\\ \hline
      \boldsymbol{0} & V' & \boldsymbol{0} \\ \hline
      0 & \boldsymbol{0} & 1
    \end{array}\right] + \begin{bmatrix} 0 \\ \boldsymbol{x} \\ 0 \end{bmatrix} \boldsymbol{e}_1^T + \begin{bmatrix} 0 \\ \boldsymbol{y} \\ 0 \end{bmatrix} \boldsymbol{e}_{r+1}^T = \left[
    \begin{array}{c|c|c}
      1 & \boldsymbol{0} & 0\\ \hline
      \boldsymbol{0} & V' & \boldsymbol{0} \\ \hline
      0 & \boldsymbol{0} & 1
    \end{array}\right] + \begin{bmatrix} 0 & 0 \\ \boldsymbol{x} & \boldsymbol{y} \\ 0 & 0 \end{bmatrix} \left[ \begin{array}{c}\boldsymbol{e}_1^T \\ \hline
      \boldsymbol{e}_{r+1}^T \end{array} \right] .
  \end{equation*}
  Let us denote $ U_V \doteq \begin{bmatrix} 0 & 0 \\ \boldsymbol{x}_V & \boldsymbol{y}_V \\ 0 & 0 \end{bmatrix} $. We compute
  \begin{align*}
    V^{-1} & = \left( \left[
    \begin{array}{c|c|c}
      1 & \boldsymbol{0} & 0\\ \hline
      \boldsymbol{0} & V' & \boldsymbol{0} \\ \hline
      0 & \boldsymbol{0} & 1
    \end{array}\right] \left( I + U_V \left[ \begin{array}{c}\boldsymbol{e}_1^T \\ \hline
      \boldsymbol{e}_{r+1}^T \end{array} \right] \right) \right)^{-1} \\
    & =
    \left( I + U_V \left[ \begin{array}{c}\boldsymbol{e}_1^T \\ \hline
      \boldsymbol{e}_{r+1}^T \end{array} \right] \right)^{-1} \left[
    \begin{array}{c|c|c}
      1 & \boldsymbol{0} & 0\\ \hline
      \boldsymbol{0} & (V')^{-1} & \boldsymbol{0} \\ \hline
      0 & \boldsymbol{0} & 1
    \end{array}\right] \\
    & = \left[ I - U_V \left( I_2 + \left[ \begin{array}{c}\boldsymbol{e}_1^T \\ \hline
      \boldsymbol{e}_{r+1}^T \end{array} \right] U_V \right)^{-1} \left[ \begin{array}{c}\boldsymbol{e}_1^T \\ \hline
      \boldsymbol{e}_{r+1}^T \end{array} \right] \right] \left[
    \begin{array}{c|c|c}
      1 & \boldsymbol{0} & 0\\ \hline
      \boldsymbol{0} & (V')^{-1} & \boldsymbol{0} \\ \hline
      0 & \boldsymbol{0} & 1
    \end{array}\right] \\
    & = \left[ I - U_V \left[ \begin{array}{c}\boldsymbol{e}_1^T \\ \hline
      \boldsymbol{e}_{r+1}^T \end{array} \right] \right] \left[
    \begin{array}{c|c|c}
      1 & \boldsymbol{0} & 0\\ \hline
      \boldsymbol{0} & (V')^{-1} & \boldsymbol{0} \\ \hline
      0 & \boldsymbol{0} & 1
    \end{array}\right] \\ & = \left[
    \begin{array}{c|c|c}
      1 & \boldsymbol{0} & 0\\ \hline
      \boldsymbol{0} & (V')^{-1} & \boldsymbol{0} \\ \hline
      0 & \boldsymbol{0} & 1
    \end{array}\right] - U_V \left[ \begin{array}{c}\boldsymbol{e}_1^T \\ \hline
      \boldsymbol{e}_{r+1}^T \end{array} \right]
    = \left[
    \begin{array}{c|c|c}
      1 & \boldsymbol{0} & 0\\ \hline
      -\boldsymbol{x}_V & (V')^{-1} & -\boldsymbol{y}_V \\ \hline
      0 & \boldsymbol{0} & 1
    \end{array}\right] ,
  \end{align*}
  where the above identities are granted by the Sherman-Morrison-Woodbury formula (see \cite[Pag. $ 65 $]{Golub} and related references therein).
\end{proof}
      %

    The following result is then a straightforward consequence of Laplace's Theorem applied twice to $ V $.
    \begin{lemma} \label{lem:vdMsubmatrix}
      Under the same hypotheses of Proposition~\ref{prop:computingVanderm}, one has
      $$ \det V = \det V' . $$
    \end{lemma}
    The above Lemma shows that the determinant of the above $(r+1) \times (r+1) $ matrix can be indeed computed by considering a $ (r-1) \times (r-1) $ minor, and it is the key fact for proving the following result.
    \begin{theorem} \label{thm:detsquared}
      Let $ f^\xi_r (\theta) $ be the symbol \eqref{eq:symboldegr} associated with the Laplacian operator discretized with weights induced by points $ \{ \xi_i \}_{i=1}^{r+1} $, with $\xi_1=0$ and $\xi_{r+1}=1$. Then
      \begin{equation} \label{eq:symboldet}
        \det f^\xi_r (\theta) = \left(\det (V^{-1}) \right)^2 d_r (2 - 2 \cos(\theta)) ,
      \end{equation}
      where $d_r$ is defined in \eqref{eq:dr}.
    \end{theorem}

    \begin{proof}
    	Consider $ A^{\loc} $ and $ B^{\loc} $ and recall that they are related as in Proposition \ref{prop:bilinearchange}, Eq. \eqref{eq:changepart2}. We put
    	$$ A' \doteq [A^{\loc}_{i,j}]_{i,j=2}^{r+1} , \ A'' \doteq [A^{\loc}_{i,j}]_{i,j=2}^{r} \quad \text{and} \quad B' \doteq [B^{\loc}_{i,j}]_{i,j=2}^{r+1} , \ B'' \doteq [B^{\loc}_{i,j}]_{i,j=2}^{r}.$$
    	From \cite[Lemma $ 9 $]{GSCS}, we know that
    	$$ \det B' = \det \left[ < \ell_i', \ell_j' >\right]_{i,j=2}^{r+1} = \det \left[ < \ell_i', \ell_j' >\right]_{i,j=2}^r = d_r = \det B'' . $$
    	Let us denote
    	$$ d_r^{\xi} \doteq \det A'' = \det \left[ < \varphi_i', \varphi_j' >\right]_{i,j=2}^r .$$
	Following \eqref{eq:changepart2} and \eqref{eq:inverseVdM}, we compute 
    \begin{align*}
    	A^\loc & = V^{-T} B^\loc V^{-1} \\
    	& = \left[
    	\begin{array}{c|c|c}
    		1 & -\boldsymbol{x}_V^T & 0\\ \hline
    		\boldsymbol{0} & (V')^{-T} & \boldsymbol{0} \\ \hline
    		0 & -\boldsymbol{y}_V^T & 1
    	\end{array}\right] \left[
    	\begin{array}{c|c|c}
    	* & \boldsymbol{*} & *\\ \hline
    	\boldsymbol{*} & B'' & \boldsymbol{*} \\ \hline
    	* & \boldsymbol{*} & *
    	\end{array}\right] \left[
    	\begin{array}{c|c|c}
    	1 & \boldsymbol{0} & 0\\ \hline
    	-\boldsymbol{x}_V & (V')^{-1} & -\boldsymbol{y}_V \\ \hline
    	0 & \boldsymbol{0} & 1
    	\end{array}\right] \\
    	& = \left[
    	\begin{array}{c|c|c}
    		* & \boldsymbol{*} & *\\ \hline
    		\boldsymbol{0} & (V')^{-T} B'' & \boldsymbol{0} \\ \hline
    		* & \boldsymbol{*} & *
    	\end{array}\right] \left[
    	\begin{array}{c|c|c}
    		1 & \boldsymbol{0} & 0\\ \hline
    		-\boldsymbol{x}_V & (V')^{-1} & -\boldsymbol{y}_V \\ \hline
    		0 & \boldsymbol{0} & 1
    	\end{array}\right] \\
    	& = \left[
    	\begin{array}{c|c|c}
    		* & \boldsymbol{*} & *\\ \hline
    		\boldsymbol{*} & (V')^{-T} B'' (V')^{-1}& \boldsymbol{*} \\ \hline
    		* & \boldsymbol{*} & *
    	\end{array}\right] .
    \end{align*}
      As a consequence $ A'' = (V')^{-T} B'' (V')^{-1} $, whence $ d_r^{\xi} = \left(\det (V')^{-1} \right)^2  d_r $ and Lemma \ref{lem:vdMsubmatrix} thus immediately yields $ d_r^{\xi} = \left(\det (V^{-1}) \right)^2  d_r $. To obtain the claim of Eq. \eqref{eq:symboldet}, it is now sufficient to retrace the proof of \cite[Theorem $ 8 $]{GSCS}, whose applicability is granted by Lemma \ref{lem:differentiallagrange}.
    \end{proof}

    \subsection{Extremal eigenvalues and conditioning}\label{rem consequence}

    A remarkable feature pointed out in the analysis of equidistributed Lagrangian elements applied to the Laplacian is that eigenvalues functions of the spectral symbol are well separated; hence, there is only one of them having a zero at $ \theta=0 $  \cite{Rahla}. We exploit formula \eqref{eq:symboldet} to extend such a result to our framework.

    \begin{proposition} \label{prop:maximalminor}
      The south-east principal minor $S_\xi $ of $ f_r^\xi(\theta)$ of size $ (r-1) \times (r-1) $ has all the entries independent of $\theta$ and  $ \det S_\xi \ne 0 $. In particular, the minimal eigenvalue $m$ of $S_\xi $ is strictly positive.
    \end{proposition}

    \begin{proof}
      Let us define the south-east minor
      $$ S_\xi \doteq \left[ f_r^\xi \right]_{i,j=2}^r \in \mathbb{R}^{(r-1)\times(r-1)}.$$
      By direct inspection of relation (\ref{eq:symboldegr}) we see that all its entries are independent of $\theta$, see also Fig.~\ref{fig:Ablocks}. Since $ f_r^\xi(\theta)$ is positive semidefinite, the matrix $S_\xi$ has to be also positive semidefinite so that its determinant is nonnegative.
      We claim that $ \det S_\xi > 0 $. In fact, let us suppose by contradiction that $ \det (S_\xi) = 0 $. Then there exists an eigenvalue $ \lambda_1 = 0 $. The minor $S_\xi $ is independent of $ \theta $, hence the eigenvalues interlacing theorem implies that the eigenvalue function $ \lambda_1(f_r^\xi(\theta))$  is identically zero and the latter is in contrast with \eqref{eq:symboldet}. Hence $m=\lambda_1(S_\xi) > 0$.
    \end{proof}

    The previous result is quite important since in turn it allows to deduce directly the following theorem, which generalises \cite[Theorem $ 4 $]{Rahla} to the framework of weights.

    \begin{theorem}\label{theorem general}
      For the symbol $ f_r^\xi  (\theta)$, the following statements hold:
      \begin{enumerate}
        \item there exist constants
        $C_1,C_2>0$ (dependent on $f_r^\xi $) such that
        \begin{equation}\label{key ineqs}
          C_1 \sum_{j=1}^{r} (2-2\cos(\theta_{j})) \le
          \lambda_1(f_r^\xi ({\theta})) \le C_2 \sum_{j=1}^{r}
          (2-2\cos(\theta_{j}));
        \end{equation}
        \item there exist constants $m,M>0$ (dependent on $f_r^\xi$) such that
        \begin{equation} \label{eq:separateeig}
          0 < m \le \lambda_j(f_r^\xi ({\theta})) \le M,\ \ \ \
          j=2,\ldots,r.
        \end{equation}
        In particular $m=\lambda_1(S_\xi) > 0$ is the same quantity appearing in the last line of the proof of Proposition \ref{prop:maximalminor}.
      \end{enumerate}
    \end{theorem}
    \begin{proof}
      Given the Hermitian character of $f_r^\xi ({\theta})$ for every $\theta$, the two statements directly follow from the use of the interlacing theorem and from Proposition \ref{prop:maximalminor}, taking into account Theorem \ref{thm:detsquared}.
    \end{proof}

    \begin{remark}\label{rem:separation}
      A practical consequence of \eqref{eq:separateeig} is that graphs of the eigenvalue functions $ \lambda_i (f_r^\xi ( \theta)) $ are separated in a weak sense. For the case of equidistributed Lagrangian, it has been observed that this separation is in particular strong \cite{Rahla}. In the case of weights, this separation may be weak depending on the value of $ \xi $. In Figure \ref{fig:ploteigdeg4} we in fact compare examples of strong (left hand panel) and weak (right hand panel) separation.
    \end{remark}

    For any choice $ \{\xi_i\}_{i=1}^{r+1} $ such that the generalized Vandermonde matrix is invertible, the sequence of stiffness matrices is distributed as the symbol $f_r^\xi$. 
    Thus the union of the ranges of the eigenvalue functions
    of $f_r^\xi$ represent a cluster for their spectra, while their convex
    hull contains all the eigenvalues of the involved matrices.

    \begin{remark}\label{rem:cond}
      From Theorem \ref{theorem general}, we know that the
      minimal eigenvalue function of $f_r^\xi$ behaves as the symbol of the
      standard finite difference Laplacian, while  the other eigenvalue
      functions are well separated from zero and bounded. Furthermore,
      thanks to the analysis in \cite{marko}, the fact that the minimal
      eigenvalue of $f_r^\xi$ has a zero of order two implies that
    \begin{itemize}
        \item the minimal eigenvalue goes to zero as $n^{-2}$,
        \item the maximal eigenvalue converges from below to the maximum of the
        maximal eigenvalue function of  $f_r^\xi$,
      \end{itemize}
      and hence
      \begin{itemize}
        \item the conditioning of $A_n$ grows asymptotically exactly as $n^{2}$.
      \end{itemize}
    \end{remark}

We conclude this section with an analysis of the kernel of $f^\xi_r$.
It follows from Proposition \ref{prop:maximalminor} that there is only an eigenvalue of $f^\xi_r$ that assumes the value $ 0 $ for $ \theta = 0 $. As a consequence, 
 $ \ker \{ f_r^{\xi} (0)\} $ is generated by one vector. The following proposition shows that such a vector is independent of 
$ \{ \xi_i \}_{i=1}^{r+1} $.

\begin{proposition}
  Let $ {\bf j} \doteq (j, \ldots, j)\in\mathbb{C}^r $. For each $ \xi $ and $ r $, one has
  $$ f_r^\xi (0) { \bf 1 } = {\bf 0} .$$
\end{proposition}

\begin{proof}
  By substituting $ \theta = 0 $ in \eqref{eq:symboldegr}, one immediately obtains that
  $$ f_r^\xi (0) { \bf 1 } = \left( \alpha_\xi + \beta_\xi + \beta_\xi^T \right) { \bf 1 }.$$
  Since multiplying by $ { \bf 1 } $ equals to taking sum row-wise, we are left to prove that
  $$ \sum_{i=1}^{r} \left( \alpha_\xi + \beta_\xi + \beta_\xi^T \right)_{j,i} = 0 \quad j = 1, \ldots, r $$
  Plugging in the definition of $ \alpha_\xi $ and $ \beta_\xi $, one then has
  $$ \sum_{i=1}^{r} \left( \alpha_\xi + \beta_\xi + \beta_\xi^T \right)_{j,i} = \sum_{i=1}^r \left\langle \varphi_i', \varphi_j' \right\rangle + \left\langle \varphi_0', \varphi_j' \right\rangle = \left\langle  \sum_{i=0}^r \varphi_i', \varphi_j' \right\rangle = \left\langle 0, \varphi_j' \right \rangle = 0 . $$
  The equality $ \sum_{i=0}^r \varphi_i' = 0 $ is proved in Lemma \ref{lem:differentiallagrange}.
\end{proof}

\begin{remark}
  Since $ \dim \ker \{ f_r^{\xi} (0)\} = 1 $, as an immediate corollary we obtain that $ \ker \{ f_r^{\xi} (0)\} = \left \langle {\bf 1} \right\rangle $. We notice that the latter property is exactly the same in the case of the spectral symbol occurring with standard finite element approximations
  as in \cite[Theorem $4$, item $1$]{Rahla}.
\end{remark}

\section{Optimization of weights}\label{sec:pb}

We apply the spectral theory of Section \ref{sec:spectral} to the case of weights introduced in Section \ref{ssec:weights}. As a first approach, in the next subsections, we confine ourselves in the case of univariate functions, although both the theory developed and the framework adopted do not impose any dimensional constraint. The multivariate case and the discontinuous case will be treated in companion papers.

\subsection{Constant coefficients} \label{ssec:constcoeff}
Let us first suppose that $ b(x) = 1 $ and that the mesh is uniform.
We now exploit the results of Section \ref{sec:spectral} and the machinery recalled in Section \ref{sec:spectral-tools}
to explicitly optimize the placement of weights
in the case $ r = 3 $ and $r=4$.

When $r=3$, a basis for weights is dual to four evaluations on $ \{\xi_i\}_{i=1}^4 $ in $ \Ihat $; this makes degree $ 3 $ the first nontrivial case.
Since global continuity requires $ \xi_1 = 0 $ and $ \xi_4 = 1 $,  we are left with the placement of $ \xi_2 $ and $ \xi_3 $. Adding the hypothesis of symmetry, one immediately gets that $ \xi_3 = 1 - \xi_2 $. We thus drop indices and consider only one parameter $ \xi $, so that nodes in $ \Ihat $ are $ \{0, \xi, 1-\xi, 1 \} $.

In our case the local  contribution of each element is constant and, recalling Equation \eqref{eq:Aloc}, we obtain that
  $$
  A^{\loc} = \begin{pmatrix}
    \frac{15\xi^4-30\xi^3+15\xi^2+2}{15\xi^2(1-\xi)^2} & \frac{4-5\xi}{30\xi^2(1-\xi)^2(2\xi-1)} & \frac{1-5\xi}{30 \xi^2(1-\xi)^2(2\xi-1)} & \frac{-30\xi^4+60\xi^3-30\xi^2+1}{30\xi^2(1-\xi)^2} \\[.2cm]
    \frac{4-5\xi}{30\xi^2(1-\xi)^2(2\xi-1)} & \frac{5\xi^2 - 5 \xi +2}{15\xi^2(1-\xi)^2(2\xi-1)^2}& \frac{10\xi^2-10\xi+1}{30\xi^2(2\xi-1)^2(1-\xi)^2} & \frac{1-5\xi}{30\xi^2(1-\xi)^2(2\xi-1)} \\[.2cm]
    \frac{1-5\xi}{30 \xi^2(1-\xi)^2(2\xi-1)} & \frac{10\xi^2-10\xi+1}{30\xi^2(2\xi-1)^2(1-\xi)^2} & \frac{5\xi^2 - 5\xi +2}{15\xi^2(2\xi-1)^2(\xi-1)^2} & \frac{4-5\xi}{30\xi^2(1-\xi)^2(2\xi-1)} \\[.2cm]
    \frac{-30\xi^4+60\xi^3-30\xi^2+1}{30\xi^2(1-\xi)^2} & \frac{1-5\xi}{30\xi^2(1-\xi)^2(2\xi-1)} & \frac{4-5\xi}{30\xi^2(1-\xi)^2(2\xi-1)} & \frac{15\xi^4 - 30\xi^3 +15\xi^2 + 2}{15\xi^2(1-\xi)^2}
  \end{pmatrix} .
  $$
  Following the construction at the beginning of Section~\ref{sec:spectral},
  the spectral symbol is the function
  \begin{align} \label{eq:symboldeg3}
    f_3^\xi: \quad [-\pi, \pi] & \to \mathbb{C}^{3 \times 3}, \\
    \theta & \mapsto \alpha_\xi + \beta_\xi e^{i \theta} + \beta_\xi^T e^{-i \theta}, \nonumber
  \end{align}
  being
  $$ \alpha_\xi = \begin{pmatrix}
    \frac{30\xi^4-60\xi^3+30\xi^2+4}{15\xi^2(1-\xi)^2} & \frac{4-5\xi}{30\xi^2(1-\xi)^2(2\xi-1)} & \frac{1-5\xi}{30 \xi^2(1-\xi)^2(2\xi-1)} \\[.2cm]
    \frac{4-5\xi}{30\xi^2(1-\xi)^2(2\xi-1)} & \frac{5\xi^2 - 5 \xi +2}{15\xi^2(1-\xi)^2(2\xi-1)^2} & \frac{10\xi^2-10\xi+1}{30\xi^2(2\xi-1)^2(1-\xi)^2} \\[.2cm]
    \frac{1-5\xi}{30 \xi^2(1-\xi)^2(2\xi-1)} & \frac{10\xi^2-10\xi+1}{30\xi^2(2\xi-1)^2(1-\xi)^2} & \frac{5\xi^2 - 5\xi +2}{15\xi^2(2\xi-1)^2(\xi-1)^2}
  \end{pmatrix} $$
  and
  $$ \beta_\xi = \begin{pmatrix}
    \frac{-30\xi^4+60\xi^3-30\xi^2+1}{30\xi^2(1-\xi)^2} & \frac{1-5\xi}{30\xi^2(1-\xi)^2(2\xi-1)} & \frac{4-5\xi}{30\xi^2(1-\xi)^2(2\xi-1)} \\[.2cm]
    0 & 0 & 0 \\[.2cm]
    0 & 0 & 0
  \end{pmatrix} .
  $$
  %

Direct computation shows that
\begin{equation} \label{eq:detsymboldeg3}
  \det(f_\xi (\theta)) = \frac{-e^{i \theta} \left(1 + 2e^{2i\theta} - 2e^{i\theta}\right)}{60\xi^4(2\xi-1)^2(\xi-1)^4} = \frac{2-2\cos(\theta)}{60\xi^4(2\xi-1)^2(\xi-1)^4},
\end{equation}
which is meaningful when $ \xi\neq0,\tfrac12,1$ (i.e. under the natural assumption that the nodes are pairwise distinct). Let us denote by $\lambda_j (f_3^\xi(\theta))$, $j=1,2,3$, the eigenvalues of 
$f_3^\xi(\theta)$ ordered nondecreasingly. 
Theorem \ref{theorem general} shows that
\begin{equation}
  \lambda_1 (f_3^\xi (0)) = 0.
\end{equation}
Moreover, Remark~\ref{rem:separation} ensures that the graphs of the the eigenvalues $\lambda_j (f_3^\xi(\theta))$, $j=1,2,3$, do not intersect for each $ \xi $, as depicted in Figures~\ref{fig:ploteigdeg3} and \ref{fig:ploteigdeg4}.
This latter Figure also shows that the separation is weak: in the right hand side panel, the minimum of the second eigenvalue is approaching the maximum of the smallest one.
Finally notice that, in contrast with the minimum of the smallest eigenvalue, which is exactly zero due to theoretical reasons (Theorem \ref{theorem general}), features of the maximum of the largest eigenvalue depends on the value of $ \xi $. This can be appreciated by comparing the panels of both Figure~\ref{fig:ploteigdeg3} and Figure~\ref{fig:ploteigdeg4}.

\begin{figure}[h]
  \centering
  \includegraphics[width=0.49\textwidth]{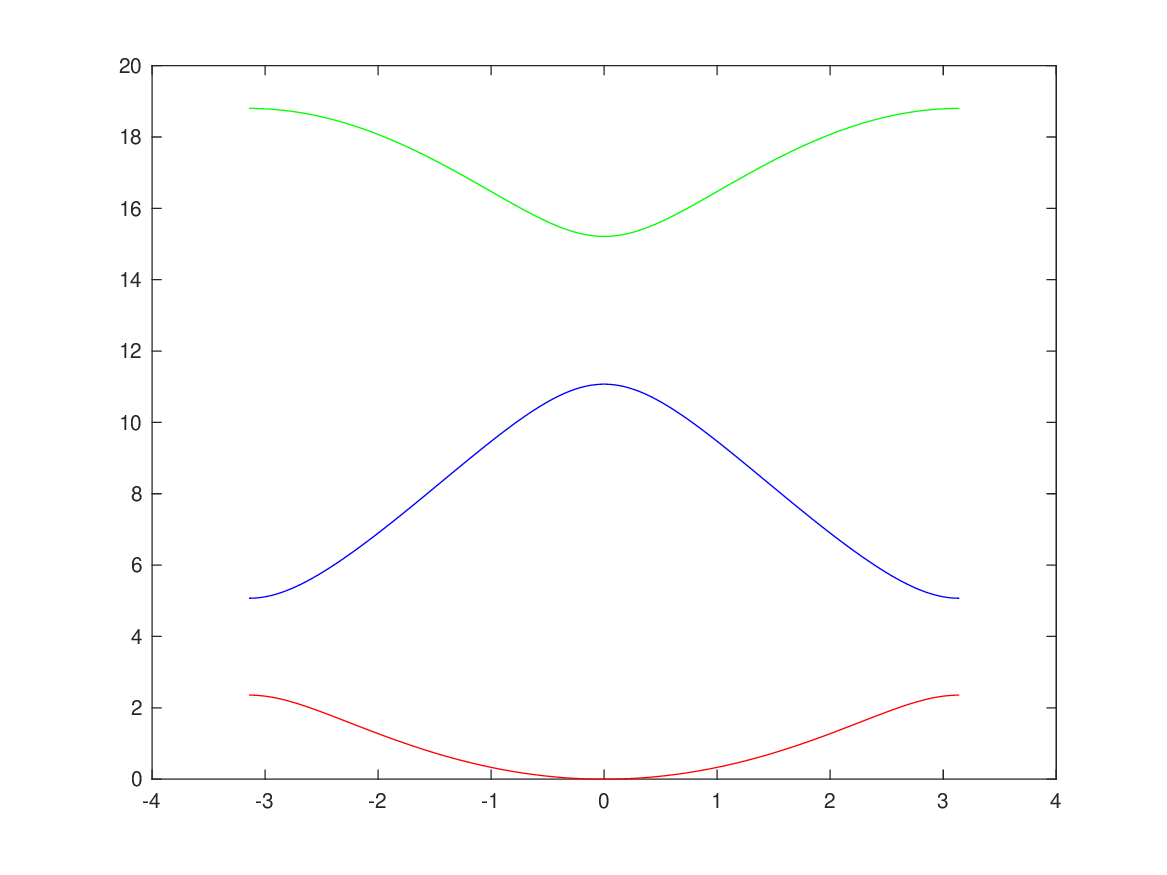}
  \includegraphics[width=0.49\textwidth]{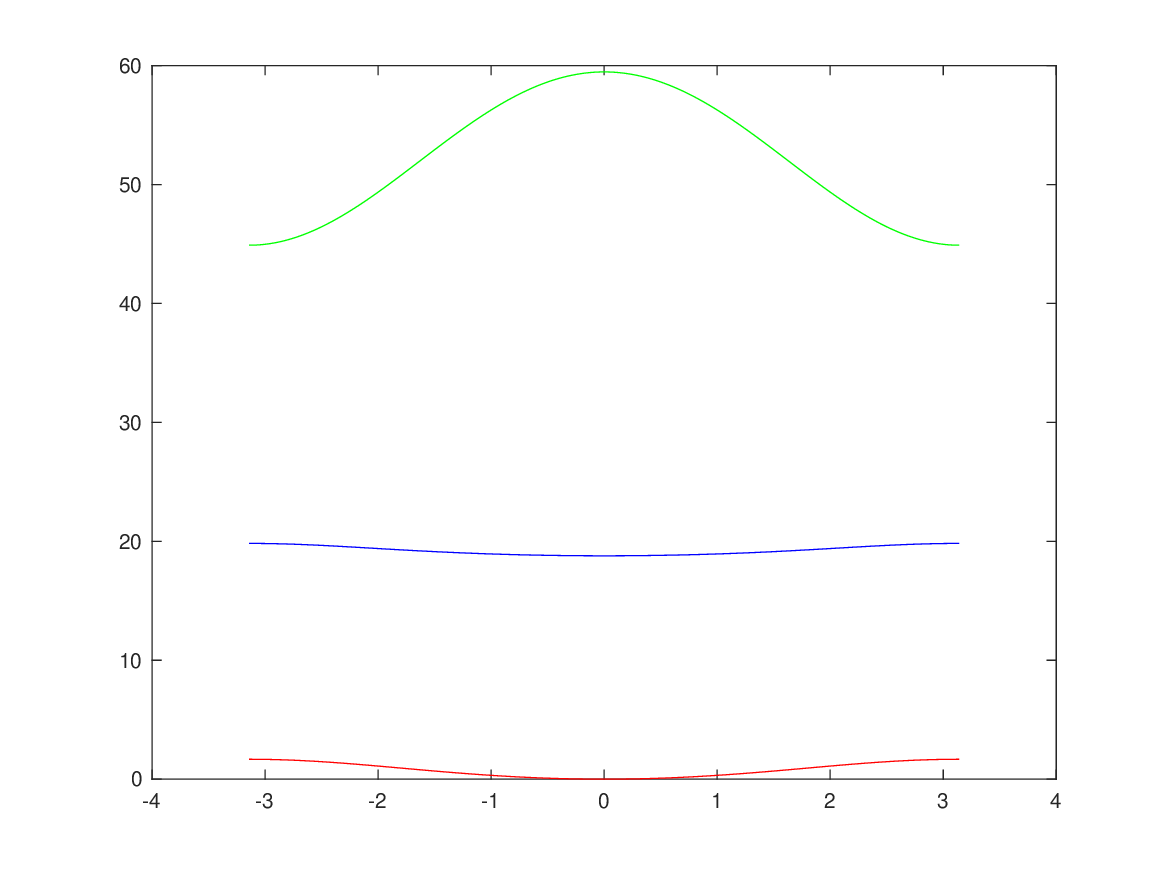}
  \caption{Eigenvalues functions of $ f_3^\xi (\theta) $ for $ \xi = 0.28 $ (left) and $\xi= 0.10 $ (right).}
  \label{fig:ploteigdeg3}
\end{figure}

The trend of $ \lambda_1(f_3^\xi(\theta))$ as $ \theta \to 0 $ is prescribed by the order of the operator and does \emph{not} depend on $ \xi $ (see Figure \ref{fig:ploteigdeg3}, red line). Expanding \eqref{eq:detsymboldeg3} around $ \theta = 0 $ we in fact retrieve that
\begin{align*} 
  \lambda_1 (f_3^\xi (\theta)) & = 
  \frac{\det(f_3^\xi(\theta))}{\lambda_2 (f_3^\xi (\theta)) \lambda_3 (f_3^\xi (\theta))} \\
  & = \theta^2 \frac{4 \xi^4 -12 \xi^3 + 13 \xi^2 - 6 \xi + 1}{3(2\xi-1)^2(\xi-1)^2} + O (\theta^4) = \frac{\theta^2}{3} + O (\theta^4) .
\end{align*}

Aiming at containing the conditioning $ \kappa_2 $ of the stiffness matrix, which is equal to the function
\begin{equation} \label{eq:condxisymbol}
  \kappa_2 (n, \xi) = n^2\frac{\max_{\theta} \lambda_3(f_3^\xi (\theta))-c_1 n^{-2}+O(n^{-3})}{c_2 \left(1+O(n^{-2})\right)}
  = \frac{n^2\max_{\theta} \lambda_3(f_3^\xi(\theta))}{c_2}\left(1+O(n^{-2})\right),
\end{equation}
with
\[
c_1=-\lambda_3^{''}(f_3^\xi(0))/2,\ \ \ c_2=\lambda_1^{''}(f_3^\xi(0))/2,
\]
we thus look for values of $ \xi $ that minimise $ \sup_{\xi} \lambda_3 $.

In contrast with $ \lambda_1 (f_r^\xi(\theta)) $, the supremum of $ \lambda_3(f_r^\xi (\theta)) $ is not necessarily attained at $ 0 $. In particular, this depends on the value of $ \xi $ (see Figure \ref{fig:ploteigdeg3}).

\begin{figure}[h]
  \centering
  \includegraphics[width=0.49\textwidth]{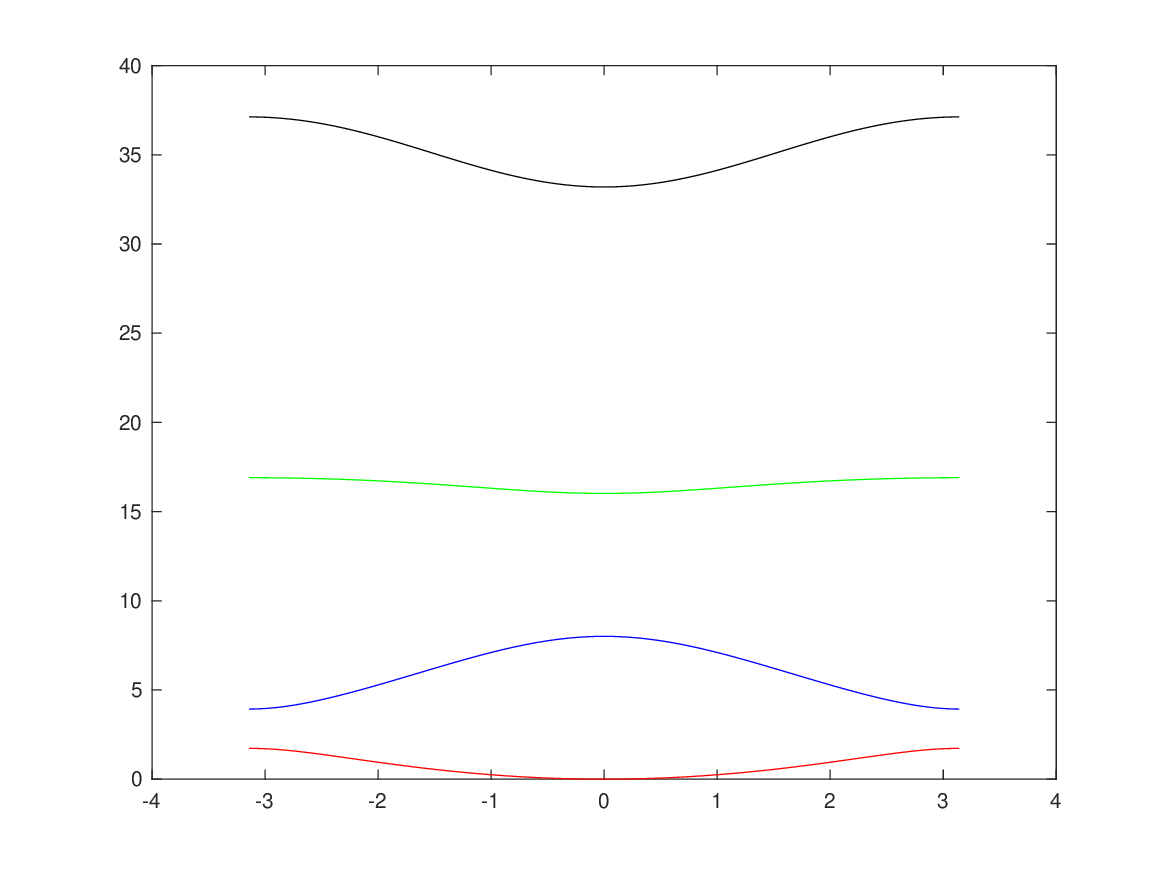}
  \includegraphics[width=0.49\textwidth]{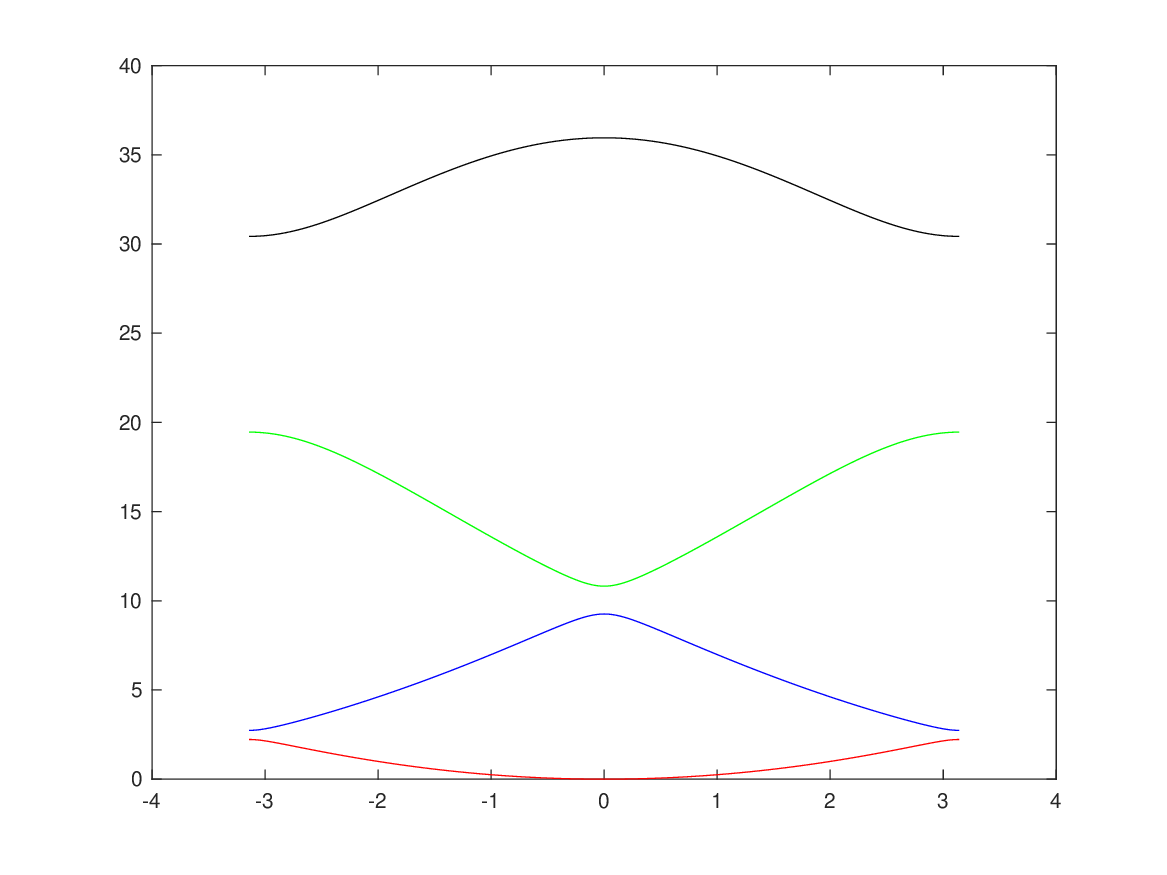}
  \caption{Eigenvalues functions of $ f_4^\xi (\theta) $ for $ \xi = 0.12 $ (left) and $ \xi = 0.22 $ (right).}
\label{fig:ploteigdeg4}
\end{figure}

As a consequence, explicit expressions for $ \lambda_3(f_r^\xi (\theta)) $ or even for $ \sup_{\xi} \lambda_3(f_r^\xi (\theta)) $ are not handy. Thus, in order to optimize the conditioning, we perform a sampling of $ \xi \in (0, \frac{1}{2}) $ and compute the quantity \eqref{eq:condxisymbol}.

\begin{figure}[h]
\centering
\includegraphics[width=0.49\textwidth]{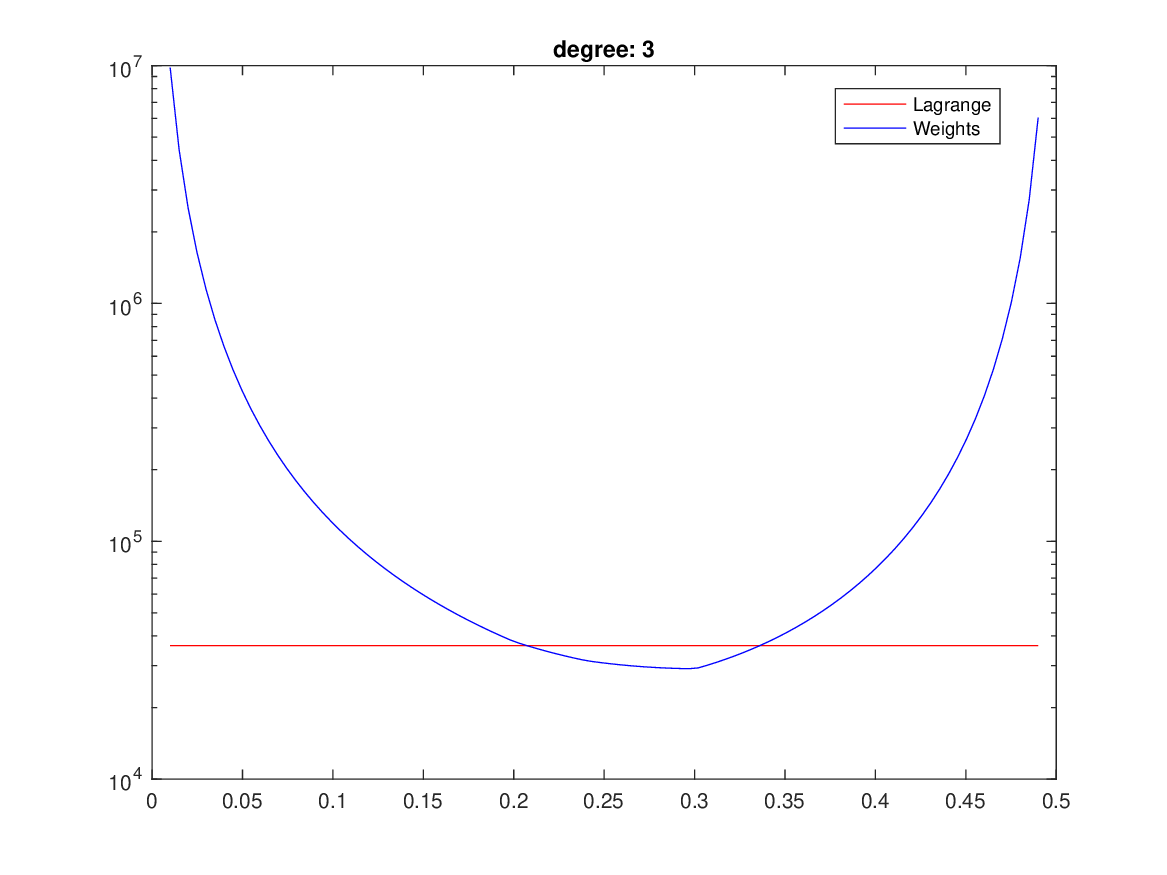}
\includegraphics[width=0.49\textwidth]{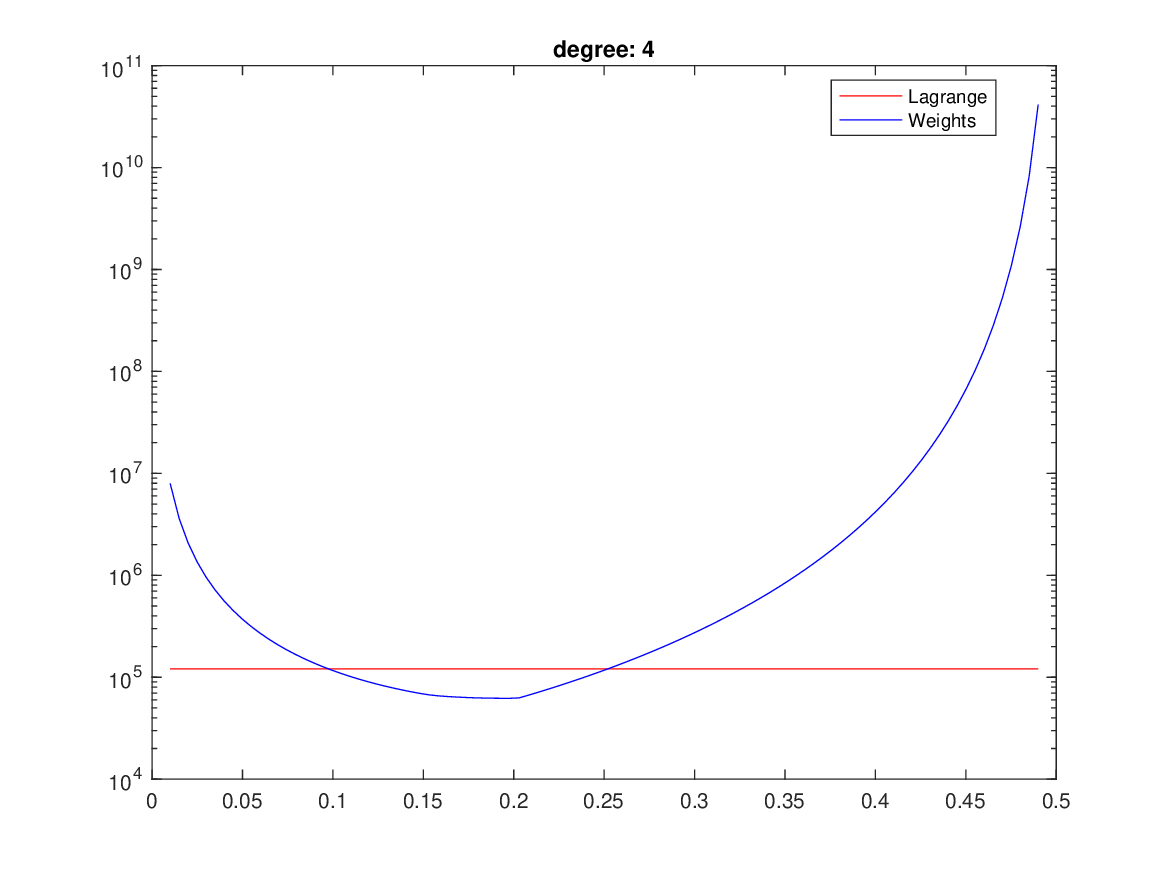}
\caption{A plot of the conditioning $ \kappa_2(n, \xi) $ as $ \xi $ varies, in logarithmic scale. Here the diameter of the interval is $ h = 1/64 $ so that the matrix size $n$ is fixed and the red line represents the conditioning with respect to uniform nodes, namely those in duality with usual Lagrangian elements. The cases $ r = 3 $ (left) and $ r = 4 $ (right) are depicted.}
\label{fig:optdeg34}
\end{figure}

The left panel of Figure \ref{fig:optdeg34} reports the quantity $ \kappa_2 (n, \xi) $ as $ \xi $ varies. As evident from the graph, there is an interval in which any choice of $ \xi $ reduces the conditioning of the spectral symbol (the blue line) and hence of the whole stiffness matrix. 
Inspecting the results, one finds that the minimum conditioning is attained at 
$ \xi \approx 0.29 $ for $ r = 3 $.

\begin{remark} \label{rmk:feketeexplanation}
  Note that, as a first attempt, one may try to optimize the conditioning of $A_n$ by minimizing the determinant of $ f_3^\xi (\theta) $. For $ r = 3 $, this gives the value $ \xi = \frac{1}{2} - \frac{1}{2 \sqrt{5}} $, thus leading to Gauss-Lobatto nodes. This is a consequence of Theorem \ref{thm:detsquared}. In fact, Gauss-Lobatto nodes maximizing the determinant of the Vandermonde matrix on an interval coincide with Fekete points \cite{Fekete}.
  We note that the value $ \xi \approx 0.29 $ is close to Gauss-Lobatto nodes.
\end{remark}

\begin{remark} \label{rmk:deg4}
A similar analysis can be produced for any degree $ r $. For $ r = 4 $ the problem can be handled again with only one parameter $ \xi $, noticing that symmetric points obey $ \left[ 0, \xi, \frac{1}{2}, 1-\xi, 1 \right] $. Again, the eigenvalue functions $\lambda_j (f_4^\xi)\theta))$, $j=1,2,3,4$, are distinct and well separated, as shown in Figure \ref{fig:ploteigdeg4}. An analysis on the optimization of the determinant yields again Gauss-Lobatto points (see Remark \ref{rmk:feketeexplanation} for an explanation), whereas a direct inspection of the graph in the right panel of Figure \ref{fig:optdeg34} shows that the optimal value is $ \xi \approx 0.21 $.
\end{remark}

\subsection{Non-constant coefficients}
Recall that in the case $b(x)=1$ the stiffness matrix $A_n$ is a principal submatrix of size $nr-1$ of the Toeplitz matrix $T_n(f_{r}^{\xi})$ with generating function $f^\xi_r$ defined in \eqref{eq:symboldegr}.
For making sizes compatible, consider
$\hat{A}_n(b)={\rm diag}(A_n(b),1)$, ${\rm order}(\hat{A}_n(b))=nr$, with
$A_n(b)$ being the stiffness matrix.
The key observation is that
\[
\hat{A}_n(b)=D_n([\sqrt b\, ] I_r)T_n(f^{r}_{\xi})D_n([\sqrt b \, ] I_r) + E^{r,b}_{\xi,n},
\]
where 
$D_n([\sqrt b\, ] I_r)$ is a diagonal matrix composed by blocks that are multiple of the $r \times r$ identity matrix; it is understood that in the $i$-th block the function $b(x)$ is evaluated at the midpoint of the $i$-th element.

By direct inspection, axiom {\bf GLT2}, item 3, and Theorem \ref{0cs}, we deduce that $\{E^{r,b}_{\xi,n}\}_n\sim_{\rm GLT} 0$.
Now, by axiom {\bf GLT2}, items 1, 2, we know that $\{T_n(f_r^\xi)\}_n\sim_{\rm GLT} f_r^\xi(\theta)$ and $\{D_n([\sqrt b\, ] I_r)\}_n\sim_{\rm GLT} [\sqrt{ b(x)}\, ] I_r$. Therefore the GLT symbol of $\{\hat{A}_n(b)\}_n$ is the Hermitian valued function defined as
\begin{equation} \label{eq:fr:b}
\begin{aligned}
f_{r,b}^{\xi}:\quad \Omega \times[-\pi,\pi] & \to \mathbb{C}^{r\times r}\\
(x,\theta) &\mapsto b(x)f_r^{\xi}(\theta),  
\end{aligned}
\end{equation}
so that $\{\hat{A}_n(b)\}_n \sim_\lambda f_{r,b}^{\xi}$ and the same is true for $\{{A}_n(b)\}_n$, thanks to Theorem \ref{extradimensional}.

\subsection{Non-uniform grids}\label{ssec:fr:quasiregular}
We start from the case of a graded mesh, in the sense of \cite{BeSe}. By this we mean a mesh that is obtained from the equispaced mesh by applying a sufficiently regular mapping.
Let us then consider a one-dimensional mesh in the domain $\Omega=I$ whose elements' endpoints are $y_k=g(x_k)$, where $x_k$ represent a uniform partition of $I$ and $g:I\to I$ is a sufficiently regular and invertible mapping that fixes the endpoints of $I$.

It is possible to observe that this case is contained in that of non constant coefficients \cite[Theorem $ 6.17 $]{etnaGLTbookIV}. 
In fact, by applying the change of variables $y=g(x)$, it can be seen that discretizing a problem with diffusion coefficient $b(x)$ on the non-uniform mesh is equivalent to discretizing on an equispaced mesh the elliptic problem with coefficient
\begin{equation} \label{eq:meshmapping}
  \widetilde{b} (x) = \frac{b(g(x))}{g'(x)} .
\end{equation}
The corresponding symbol of the sequence of stiffness matrices can thus be obtained by replacing $b$ with $\tilde{b}$ in \eqref{eq:fr:b}.

If the mesh is not graded, the present theory cannot be used as it is. Nevertheless, in the numerical section we will show that the predictions of the conditioning and the proposed preconditioners behave quite robustly under random mesh perturbations.

\section{Numerical experiments} \label{sec:numer}

Following the derivations in Section \ref{sec:spectral}, we know that the stiffness matrix has a tridiagonal Toeplitz-like structure with blocks of size $r$. The related generating function, given in equation (\ref{eq:symboldegr}) or \eqref{eq:fr:b}, is a Hermitian-valued, nonnegative definite linear trigonometric polynomial, whose minimal eigenvalue has a unique zero at zero of order $2$. Hence by exploiting known results in the literature \cite{marko}, as indicated in Remark \ref{rem:cond}, we know that the conditioning grows as $c_r^\xi n^2$ where the constant $c_r^\xi$ depends essentially on two analytic computable quantities, i.e. $\max_{\theta} \lambda_{\max}(f_r^\xi(\theta))$ and the second derivative of $\lambda_{\min}(f_r^\xi (\theta))$ at $\theta=0$, according to Theorem \ref{theorem general}. Figure \ref{fig:optdeg34} shows as the quantity can be affected by the constant $c_r^\xi$ and as it can be minimized.

According the previous spectral observations, from the point of view of the solution of a linear system of such a kind, two basic techniques can be proposed.

  %

\begin{itemize}
  \item[a1.] A preconditioned conjugate gradient (PCG) method where the preconditioner is a block circulant matrix with Strang correction having the same GLT symbol as the original matrix-sequence. When the problem is with a variable coefficient diffusion term, the preconditioner is enriched with the use of block diagonal sampling matrix so that, again, the GLT symbol of the precondining matrix-sequence coincides with that of the original matrix-sequence and hence by axiom {\bf GLT1} and items 3, 4 of axiom {\bf GLT3}, we deduce that the preconditioned matrix-sequence has symbol 1, the latter meaning that all the eigenvalues are clustered at $1$, according to Definition \ref{def-cluster} and Remark \ref{rem:clustering vs distribution}, so indicating a fast convergence.
  \item[a2.] A multigrid method which is essentially standard with a tensorization of the standard projection matrix. This is based on the inequality $L_n \le c T_n$ with $L_n=\Delta_n\otimes I_r$ and $\Delta_n$ being the standard discrete Laplacian obtained using centered Finite Differences i.e. $\Delta_n=T_n(g(\theta))$, $g(\theta)=2-2\cos(\theta))$.
\end{itemize}

In this section we follow the approach a1 and 
specifically we propose a circulant preconditioner matrix; we point out that the choice of circulant matrices allows to apply the preconditioner efficiently via FFT transforms.
In particular, we use a $r$-block circulant preconditioner of size $nr$; since our stiffness matrix $A_n$ has size $nr-1$, we change the linear system by adding a fictitious equation so that the resulting coefficient matrix is
\[
\hat{A}_n={\rm diag}(A_n,1),\ \ \ {\rm order}(\hat{A}_n)=nr.
\]
The circulant preconditioner for the $b(x)=1$ case is then constructed as $S_n=C_n(f_r^\xi)+\frac{i}{nr}{ \bf 1 }{ \bf 1 }^T$, since $C_n(f_r^\xi)$ is singular with null space of dimension $1$ generated by the vector ${ \bf 1 }$: the matrix $S_n$ is known as the corrected Strang preconditioner of the Toeplitz matrix $T_n(f_r^\xi)$ (see \cite{Ng-book} and references therein).
It is nice to observe that $S_n=\hat{A}_n+R_n$ where all the matrices are real symmetric (hence Hermitian) and rank$(R_n)\le 2r+2$. Therefore by axiom
{\bf GLT2}, item 3, and using the first part of Theorem \ref{0cs}, we deduce that $\{R_n\}_n\sim_{\rm GLT} 0$ so that by axiom {\bf GLT2}, item 3, the three matrix-sequences $\{T_n(f_r^\xi)\}_n$, $\{\hat{A}_n\}_n$, $\{S_n\}_n$ have the same GLT symbol $f_r^\xi$. Therefore, by exploiting axiom {\bf GLT3}, items 3, 4, we easily conclude that the preconditioned matrix-sequence has GLT symbol equal to $1$ and hence it is weakly clustered at $1$ in a spectral sense. Indeed in this specific setting more is known since the preconditioned matrix minus the identity has rank bounded by $2r-2$ and hence a constant number of iterations has to be expected.

\subsection{Convergence of the FEM method}

First of all, we check that the convergence of the FEM Galerkin discretization is in accordance with the theory. We show this for the case of constant coefficients and consider a uniform mesh; more elaborated distributions are also taken into account and discussed from Section \ref{sect:quasiregularmesh} on. 
It is well known that, if the datum is sufficiently regular, then the $H^1$-approximation error given by the finite element method for this kind of problems decays as $ h^{r} $, being $ h $ the element size and $ r $ the polynomial degree (see, e.g., \cite{CiarletBook}). It is important to point out that, being this a Galerkin method, the error shall not depend on the projector but only on the discrete space in which the solution is sought. As a consequence, we shall also check that (up to rounding errors) the convergence for equidistributed Lagrangian and weights coincide, since in both cases they seek for a solution in the subspace of fucntions in $H^1$ which are locally polynomials of degree $r$ in each element. Tables \ref{tab:errorsdeg3} and \ref{tab:errorsdeg4} show that the theoretical convergence order is attained, at least until computations start being affected by rounding errors.

\begin{table}[!h]
  \centering
  \begin{tabular}{|c|c|c|c|c|}
    \hline
    \multicolumn{1}{|c|}{} & \multicolumn{2}{|c|}{error} & \multicolumn{2}{|c|}{convergence rate} \\
    \hline
    elements & Lagrangian & Weights & Lagrangian & Weights \\
    \hline
    $ 10 $ & $ 5.26 \times 10^{-4} $ & $ 5.26 \times 10^{-4} $ & &  \\
    $ 20 $ & $ 6.58 \times 10^{-5} $ & $ 6.58 \times 10^{-5} $ & $ 2.9991 $ & $ 2.9991 $ \\
    $ 40 $ & $ 8.22 \times 10^{-6} $ & $ 8.22 \times 10^{-6} $ & $ 2.9998 $ & $ 2.9998 $ \\
    $ 80 $ & $ 1.03 \times 10^{-6} $ & $ 1.03 \times 10^{-6} $ & $ 2.9999 $ & $ 2.9999 $  \\
    $ 160 $ & $ 1.29 \times 10^{-7} $ & $ 1.31 \times 10^{-7} $ & $ 3.0000 $ & $ 2.9640 $ \\
    $ 320 $ & $ 1.73 \times 10^{-8} $ & $ 1.61 \times 10^{-8} $ & $ 2.8932 $ & $ 3.0361 $ \\
    $ 640 $ & $ 2.02 \times 10^{-9} $ & $ 2.02 \times 10^{-9} $ & $ 3.0982 $ & $ 2.9932 $ \\
    $ 1280 $ & $ 6.80 \times 10^{-10} $ & $ 1.23 \times 10^{-9} $ & $ 1.5705 $ & $ 0.7160 $ \\
    \hline
  \end{tabular}
  \caption{Approximation error and convergence rate for $ r = 3 $.} \label{tab:errorsdeg3}
\end{table}

\begin{table}[!h]
  \centering
  \begin{tabular}{|c|c|c|c|c|}
    \hline
    \multicolumn{1}{|c|}{} & \multicolumn{2}{|c|}{error} & \multicolumn{2}{|c|}{convergence rate} \\
    \hline
    elements & Lagrangian & Weights & Lagrangian & Weights \\
    \hline
    $ 2 $ & $ 5.72 \times 10^{-3} $ & $ 5.72 \times 10^{-3} $ & &  \\
    $ 4 $ & $ 3.58 \times 10^{-4} $ & $ 3.58 \times 10^{-4} $ & $ 3.9993 $ & $ 3.9993 $ \\
    $ 8 $ & $ 2.25 \times 10^{-5} $ & $ 2.25 \times 10^{-5} $ & $ 3.9931 $ & $ 3.9931 $ \\
    $ 16 $ & $ 1.41 \times 10^{-6} $ & $ 1.41 \times 10^{-6} $ & $ 3.9980 $ & $ 3.9966 $  \\
    $ 32 $ & $ 8.79 \times 10^{-8} $ & $ 1.17 \times 10^{-7} $ & $ 3.9995 $ & $ 3.5856 $ \\
    $ 64 $ & $ 9.74 \times 10^{-9} $ & $ 1.37 \times 10^{-8} $ & $ 3.1738 $ & $ 3.0959 $ \\
    $ 128 $ & $ 3.57 \times 10^{-10} $ & $ 3.53 \times 10^{-10} $ & $ 4.7704 $ & $ 5.2782 $ \\
    $ 256 $ & $ 3.58 \times 10^{-10} $ & $ 3.39 \times 10^{-10} $ & $ -0.0046 $ & $ 0.0611 $ \\
    \hline
  \end{tabular}
  \caption{Approximation error and convergence rate for $ r = 4 $.} \label{tab:errorsdeg4}
\end{table}

In the following sections we will compare the spectrum of the assembled matrices with that predicted by the generating symbol and study the effectiveness of the proposed circulant preconditioner.

\subsection{Constant coefficients} \label{sec:expconstcoeff}

In Figure \ref{fig:spectrumdeg3}, for $r=3$, we show that the spectrum of the operator (blue, solid line) is in fact captured by the spectral symbol (as the number of elements increase, the color of the corresponding dashed line changes). This is in accordance with the theory reported and discussed in the previous sections. Moreover, we see that $r$ distinct branches are present. The latter holds both for the case of equispaced Lagrangian (left) and optimized weights (right).
Notice that the magnitude of the eigenvalues changes.
Analogous results were observed in the case $r=4$ (not shown).

\begin{figure}[h]
  \centering
  \includegraphics[width=0.49\textwidth]{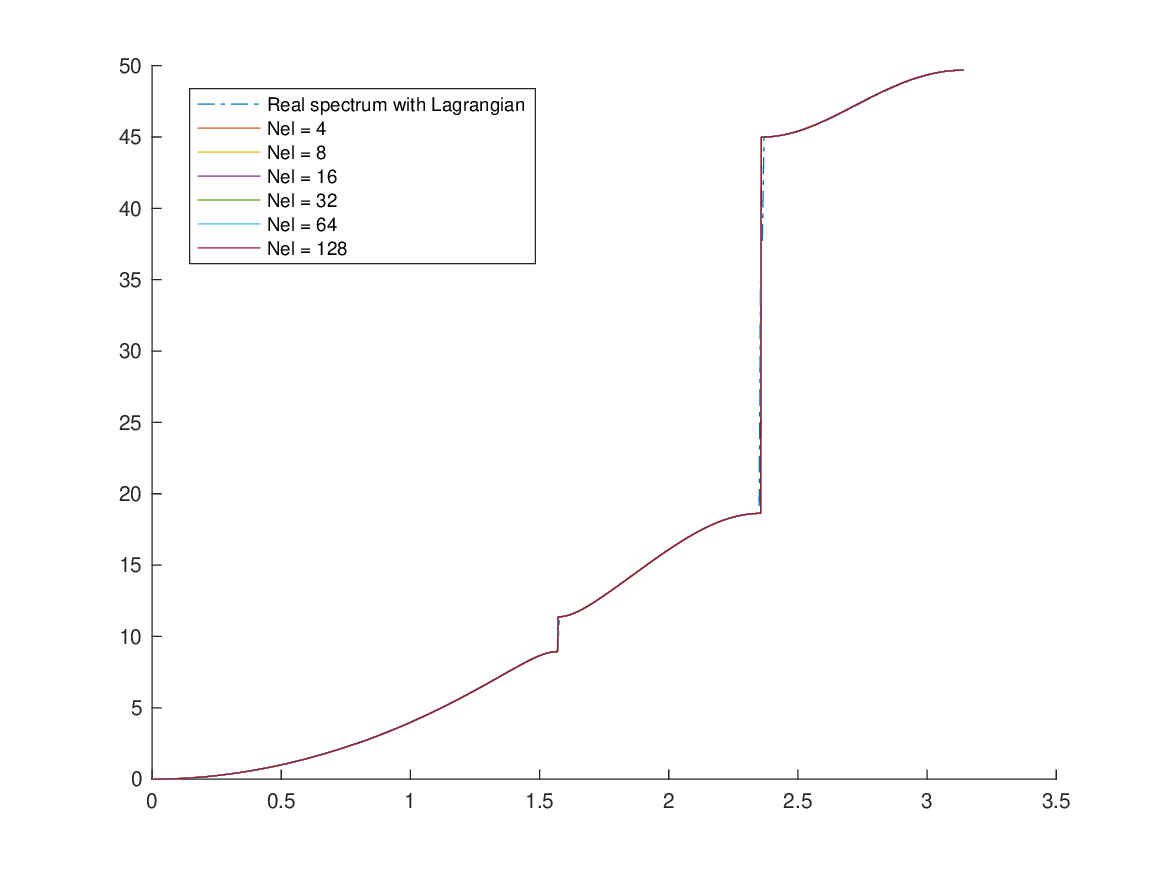}
  \includegraphics[width=0.49\textwidth]{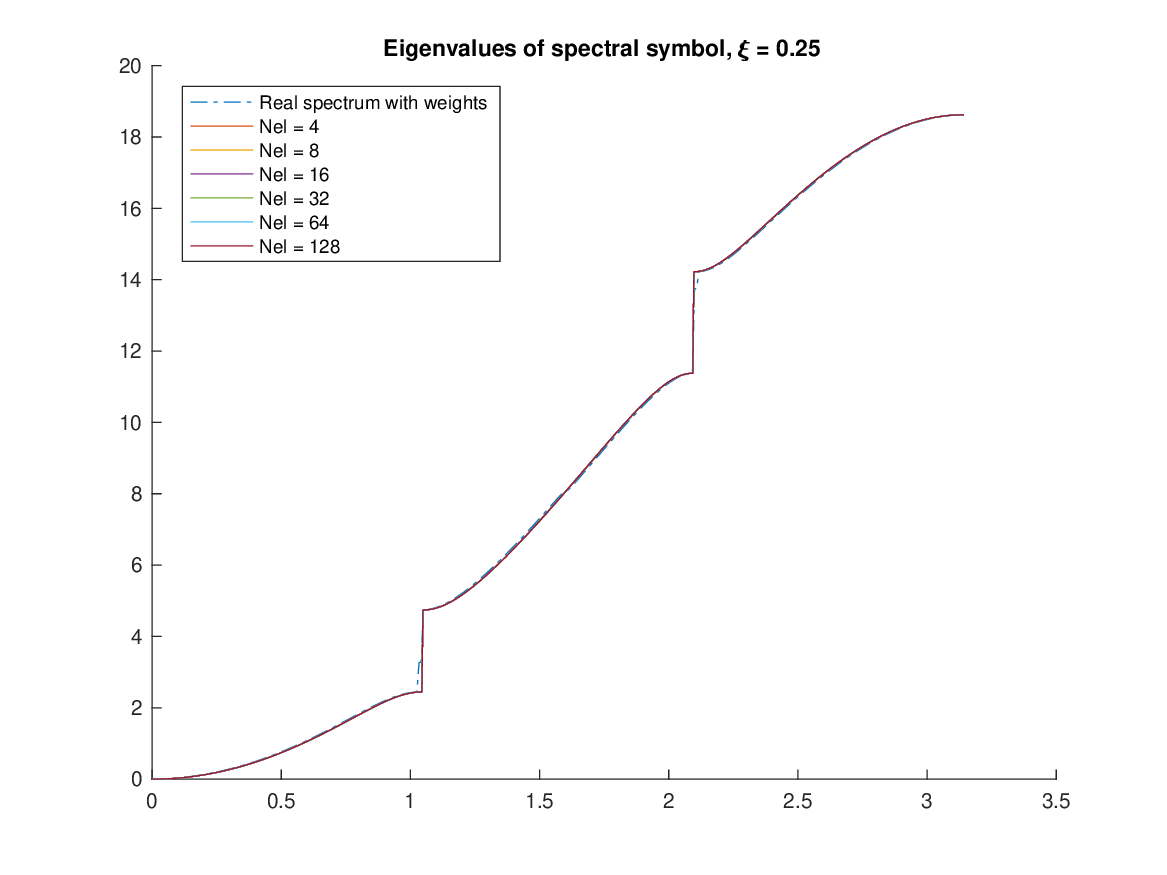}
  \caption{A plot of the spectrum of the Laplacian with constant coefficients. Solid line represents the continuous spectrum, dashed lines represent the symbol associated with different mesh sizes. Lagrange equispaced elements (left) and optimized weights (right) are compared. The degree $ r = 3 $ is depicted.}
  \label{fig:spectrumdeg3}
\end{figure}

In the first experiment, we compare the conditioning of the stiffness matrices and the observe the optimality of the proposed preconditioning strategy. Results are reported in Table \ref{tab:UniformMeshdeg3} for $ r = 3 $ and Table \ref{tab:UniformMeshdeg4} for $ r = 4 $.
Tables show that a slight improvement in the conditioning yields in turn a significant reduction in the number of iterations when the preconditioned conjugate gradients is applied. For both the considered degrees, this can be appreciated in terms of runtime, as shown in Figure \ref{fig:runtime}.

\begin{table}[h]
  \centering
  \begin{tabular}{|c|c|c|c|c|}
    \hline
    \multicolumn{1}{|c|}{} & \multicolumn{2}{|c|}{Conditioning} & \multicolumn{2}{|c|}{Iterations} \\
    \hline
    elements & Lagrangian & weights & Lagrangian & weights \\
    \hline
    $ 10 $ & $ 9.11 \times 10^2 $ & $ 7.75 \times 10^2 $ & $ 7 $ & $ 4 $ \\
    $ 20 $ & $ 3.65 \times 10^3 $ & $ 3.10 \times 10^3 $ & $ 9 $ & $ 4 $ \\
    $ 40 $ & $ 1.46 \times 10^4 $ & $ 1.24 \times 10^4 $ & $ 9 $ & $ 4 $ \\
    $ 80 $ & $ 5.83 \times 10^4 $ & $ 4.96 \times 10^4 $ & $ 10 $ & $ 5 $ \\
    $ 160 $ & $ 2.33 \times 10^5 $ & $ 1.98 \times 10^5 $ & $ 11 $ & $ 5 $ \\
    $ 320 $ & $ 9.33 \times 10^5 $ & $ 7.93 \times 10^5 $ & $ 11 $ & $ 6 $ \\
    $ 640 $ & $ 3.73 \times 10^6 $ & $ 3.17 \times 10^6 $ & $ 11 $ & $ 6 $ \\
    $ 1280 $ & $ 1.49 \times 10^7 $ & $ 1.27 \times 10^7 $ & $ 11 $ & $ 6 $ \\
    \hline
  \end{tabular}
  \caption{Comparison of the conditioning of the stiffness matrix and number of iterations needed by conjugate gradients after preconditioning via Circulant matrices for $ r = 3 $.}\label{tab:UniformMeshdeg3}
\end{table}

\begin{table}[h]
  \centering
  \begin{tabular}{|c|c|c|c|c|}
    \hline
    \multicolumn{1}{|c|}{} & \multicolumn{2}{|c|}{Conditioning} & \multicolumn{2}{|c|}{Iterations} \\
    \hline
    elements & Lagrangian & weights & Lagrangian & weights \\
    \hline
    $ 2 $ & $ 1.14 \times 10^2 $ & $ 9.06 \times 10^1 $ & $ 8 $ & $ 4 $ \\
    $ 4 $ & $ 4.81 \times 10^2 $ & $ 3.71 \times 10^2 $ & $ 15 $ & $ 4 $ \\
    $ 8 $ & $ 1.92 \times 10^3 $ & $ 1.48 \times 10^3 $ & $ 15 $ & $ 4 $ \\
    $ 16 $ & $ 7.70 \times 10^3 $ & $ 5.94 \times 10^3 $ & $ 15 $ & $ 4 $ \\
    $ 32 $ & $ 3.08 \times 10^4 $ & $ 2.38 \times 10^4 $ & $ 16 $ & $ 4 $ \\
    $ 64 $ & $ 1.23 \times 10^5 $ & $ 9.50 \times 10^4 $ & $ 16 $ & $ 5 $ \\
    $ 128 $ & $ 4.93 \times 10^5 $ & $ 3.80 \times 10^5 $ & $ 18 $ & $ 6 $ \\
    $ 256 $ & $ 1.97 \times 10^6 $ & $ 1.52 \times 10^6 $ & $ 18 $ & $ 6 $ \\
    \hline
  \end{tabular}
  \caption{Comparison of the conditioning of the stiffness matrix and number of iterations needed by conjugate gradients after preconditioning via Circulant matrices for $ r = 4 $.}\label{tab:UniformMeshdeg4}
\end{table}

\begin{figure}[h]
  \centering
  \includegraphics[width = 0.49\textwidth]{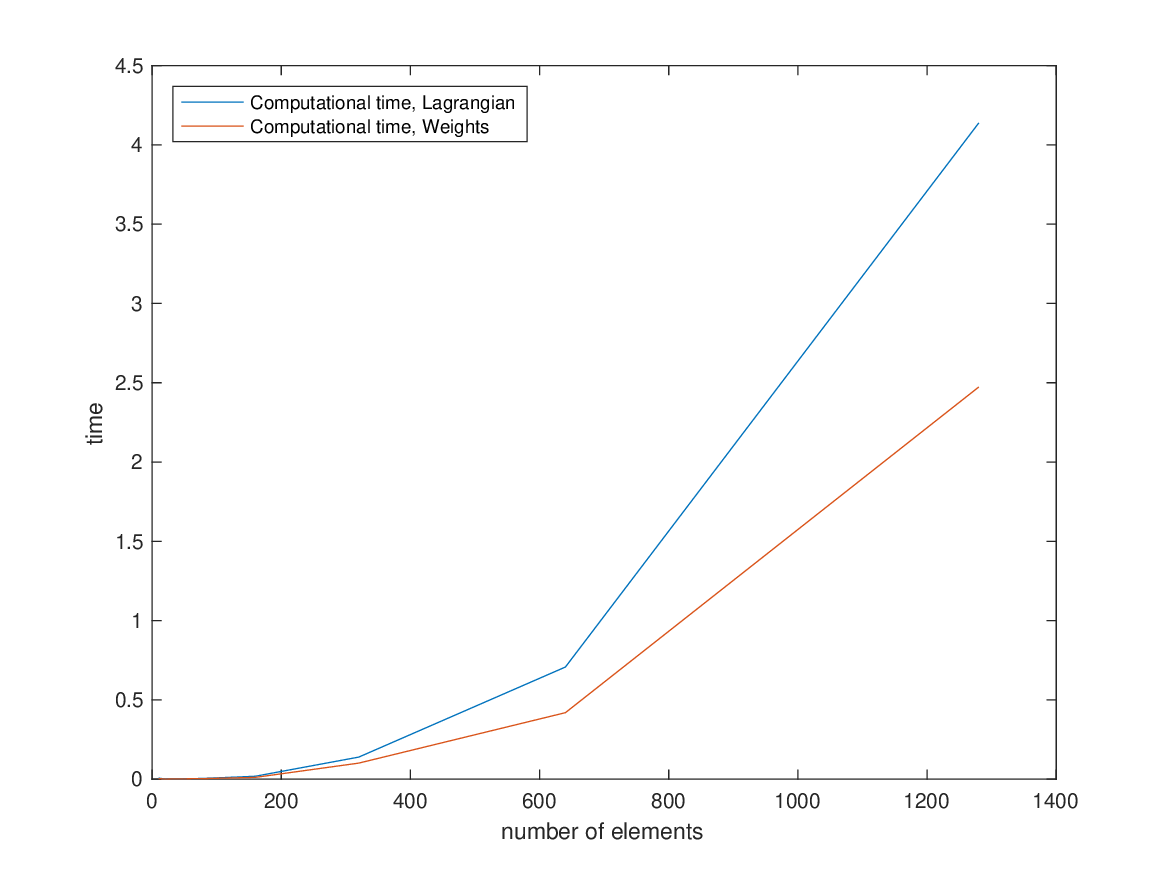} \includegraphics[width = 0.49\textwidth]{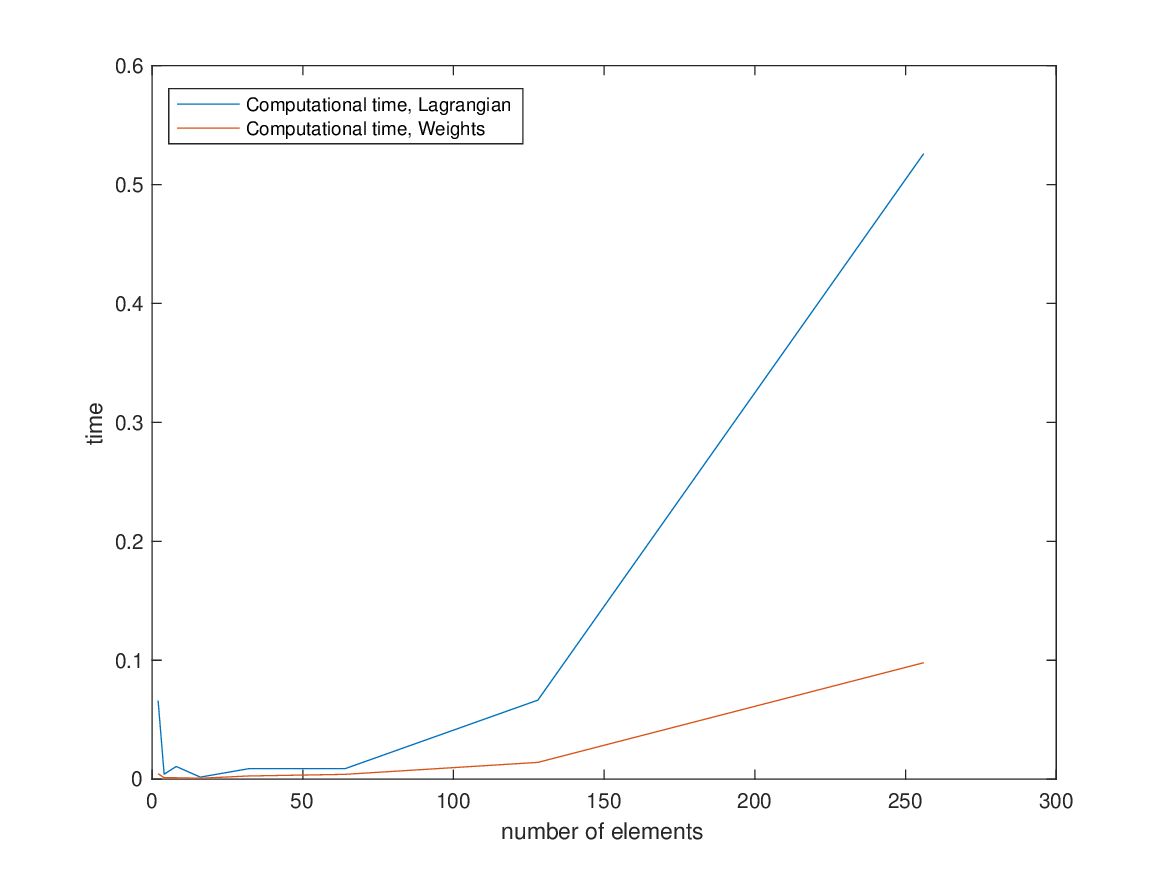}
  \caption{A comparison between equidistributed Lagrangian and optimised weights in terms of runtime.}
  \label{fig:runtime}
\end{figure}

\subsection{Non-constant coefficients}\label{sec:expnonconstcoeff}

In this section we repeat the same analysis done in Section \ref{sec:expconstcoeff}, in the case of non-constant coefficient $ b(x) $. For the sake of simplicity, we present an example in which $ b(x) = 1 + x^2 $, although GLT theory applies in the case of a function $b(x)$ Riemann integrable over its definition domain.
Figure \ref{fig:nonconstantcoeff} shows that the spectrum is captured by the spectral symbol.

\begin{figure}[!h]
  \includegraphics[width = 6cm]{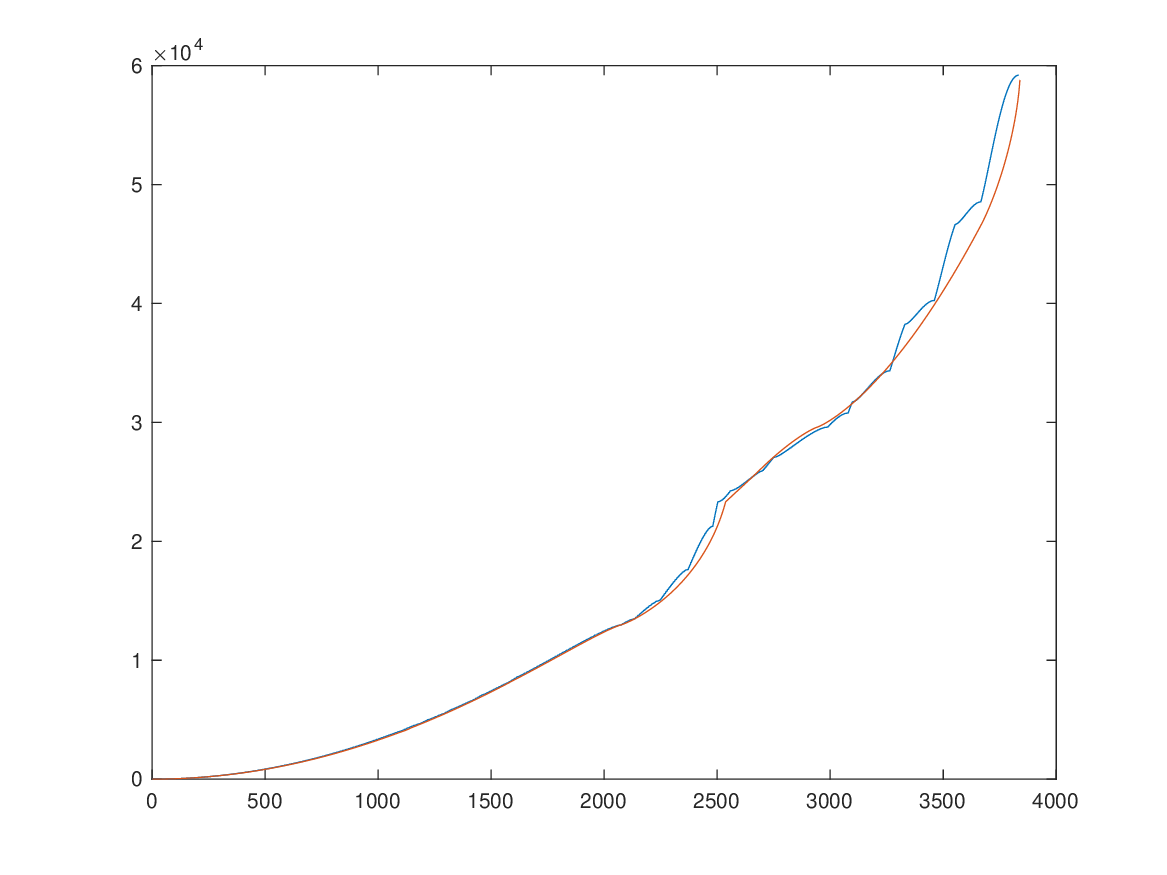} \quad \includegraphics[width = 6cm]{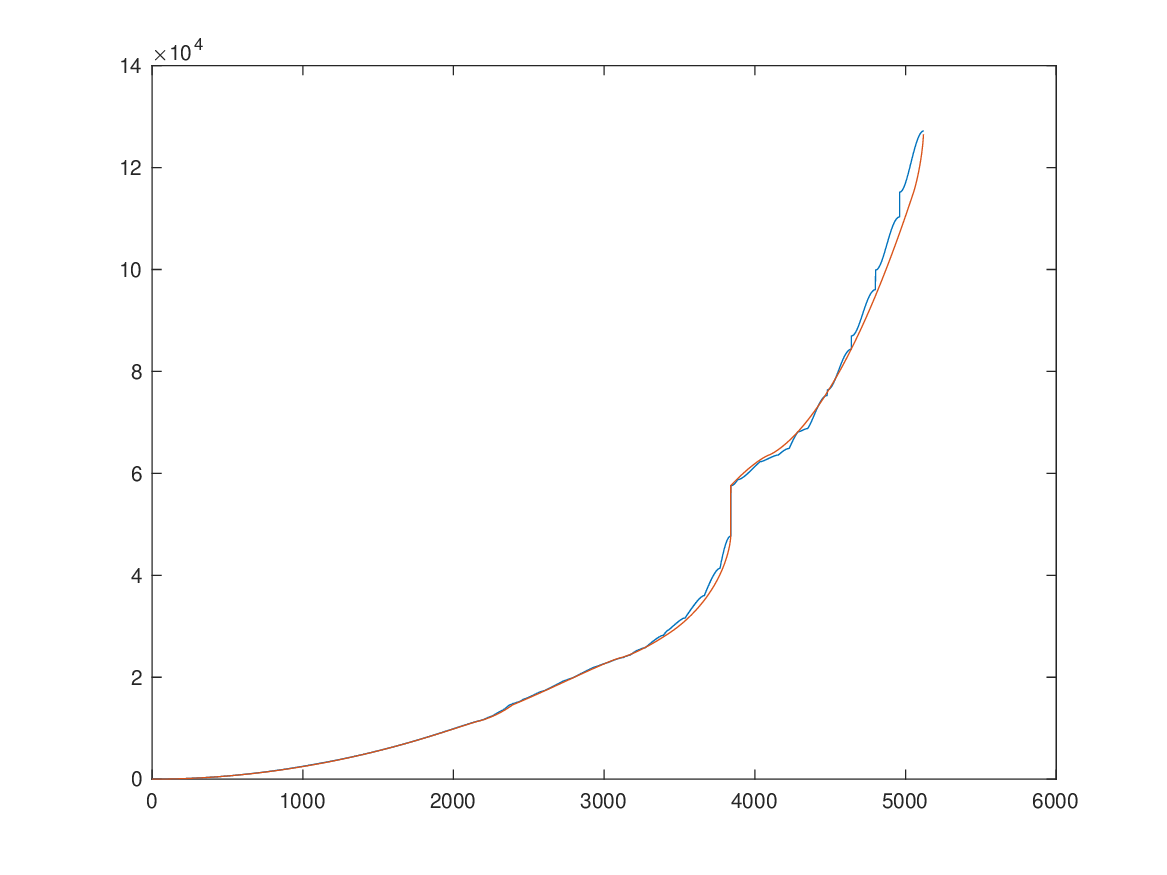} \\
  \includegraphics[width = 6cm]{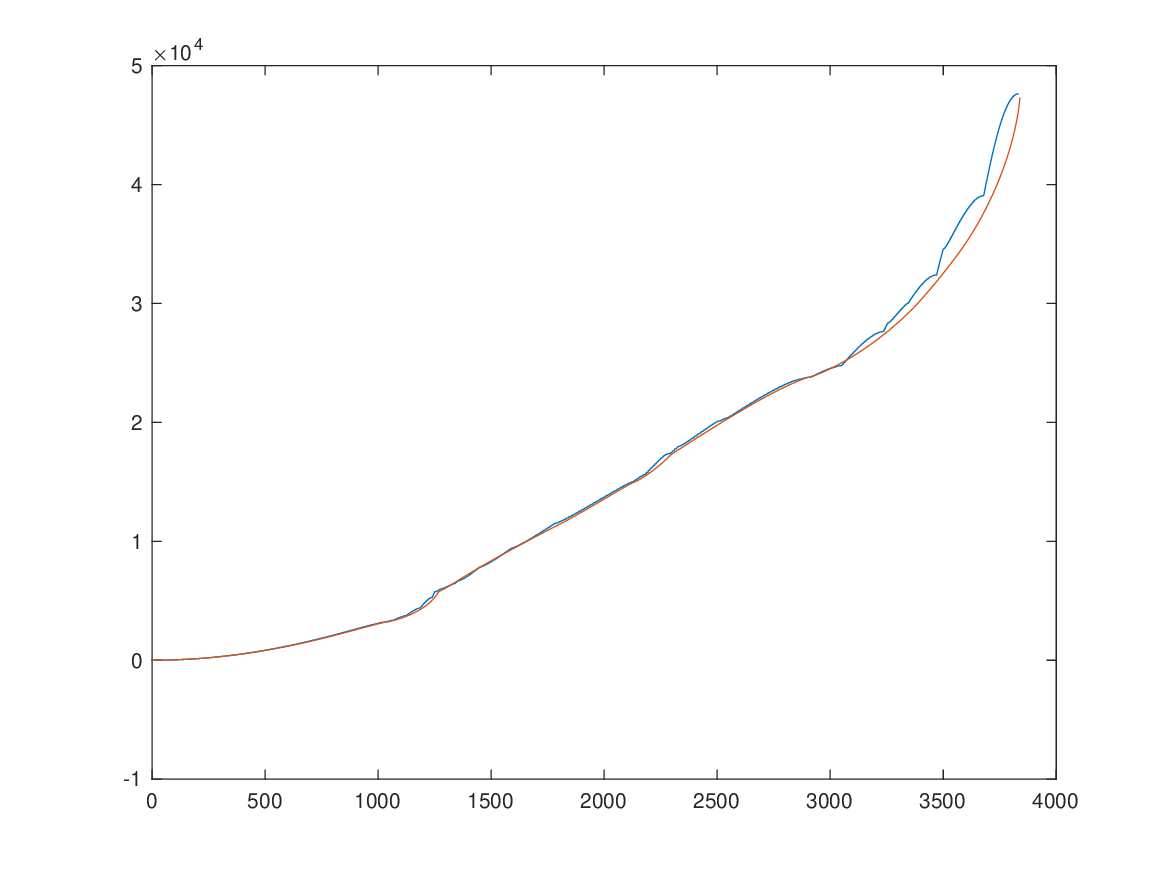} \quad \includegraphics[width = 6cm]{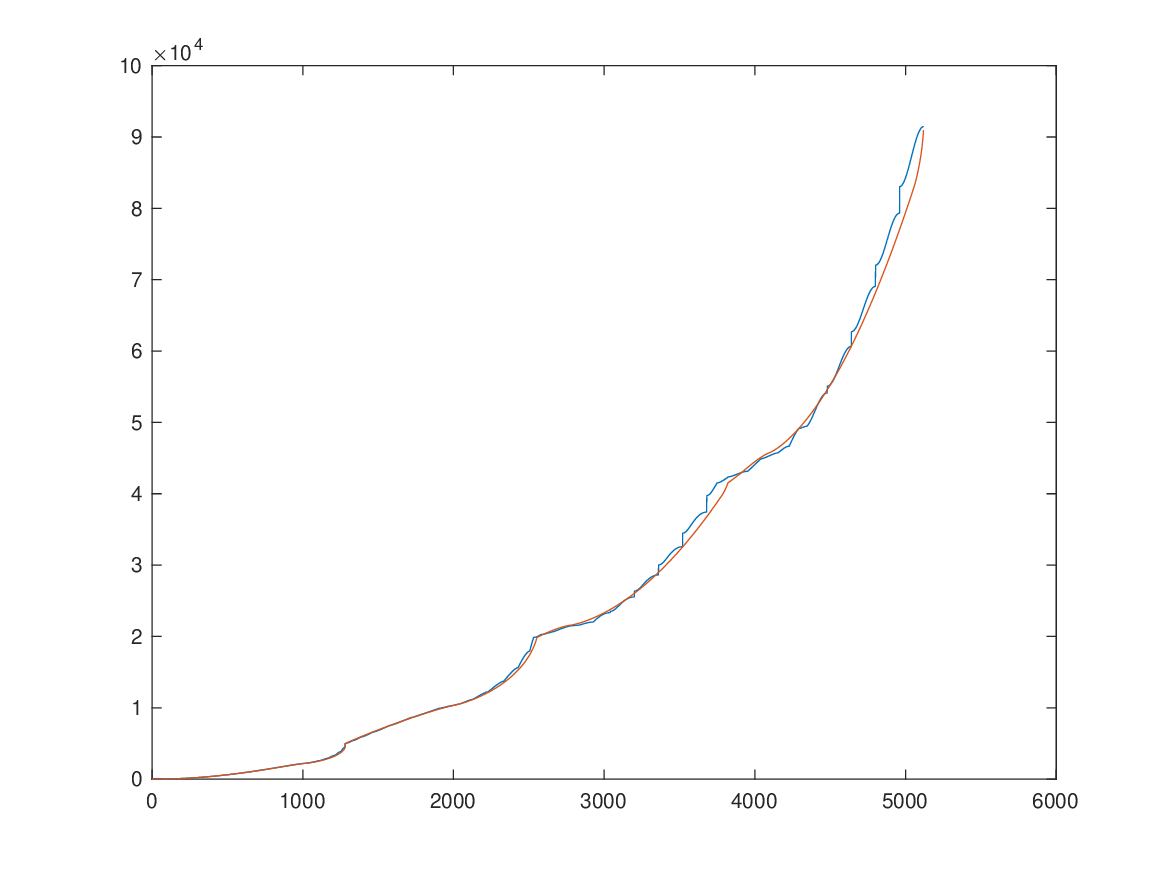}
  \caption{Spectrum of the operator (red) vs spectrum of the symbol (blue), in the case of non constant coefficients with $ a(x) = 1+x^2 $. Degree 3 (left) and degree 4 (right), for equidistrubuted Lagrangian (above) and optimized weights (bottom).} \label{fig:nonconstantcoeff}
\end{figure}

Clearly, non constant coefficients make it harder to construct an ad hoc preconditioner via Toeplitz or Circulant matrices, since elements do not scale only by a constant but accordingly to the coefficient $ b(x) $ as well.
Here we consider the Diagonal plus Circulant preconditioner with Strang correction based on the symbol $f_{r,b}^{\xi}(x,\theta)$ introduced above and defined as
\begin{equation} \label{eq:Prec:b}
P_n(b)=D_n([\sqrt b\,] I_r) S_nD_n([\sqrt b\,] I_r).
\end{equation}
By following the very same reasoning as before, we infer that  $\{\hat{A}_n(a)\}_n$ and $\{P_n(a)\}_n$ have the same GLT symbol $f^{r,a}_{\xi}$. Therefore, by exploiting axiom {\bf GLT3}, items 3, 4, we easily conclude that the preconditioned matrix-sequence has GLT symbol equal to $1$ and hence it is weakly clustered at $1$ in a spectral sense, the latter being an indication that the related preconditioned conjugate gradient is effective.

\begin{table}[h]
  \centering
  \begin{tabular}{|c|c|c|c|c|}
    \hline
    \multicolumn{1}{|c|}{} & \multicolumn{2}{|c|}{Conditioning} & \multicolumn{2}{|c|}{Iterations} \\
    \hline
    elements & Lagrangian & weights & Lagrangian & weights \\
    \hline
    $ 10 $ & $ 1.22 \times 10^3 $ & $ 1.05 \times 10^3 $ & $ 13 $ & $ 13 $ \\
    $ 20 $ & $ 5.14 \times 10^3 $ & $ 4.43 \times 10^3 $ & $ 14 $ & $ 14 $ \\
    $ 40 $ & $ 2.13 \times 10^4 $ & $ 1.82 \times 10^4 $ & $ 15 $ & $ 14 $  \\
    $ 80 $ & $ 8.66 \times 10^4 $ & $ 7.36 \times 10^4 $ & $ 16 $ & $ 15 $ \\
    $ 160 $ & $ 3.49 \times 10^5 $ & $ 2.96 \times 10^5 $ & $ 17 $ & $ 15 $ \\
    $ 320 $ & $ 1.40 \times 10^6 $ & $ 1.19 \times 10^6 $ & $ 18 $ & $ 16 $ \\
    $ 640 $ & $ 5.63 \times 10^6 $ & $ 4.78 \times 10^6 $ & $ 18 $ & $ 16 $ \\
    $ 1280 $ & $ 2.25 \times 10^7 $ & $ 1.91 \times 10^7 $ & $ 19 $ & $ 18 $ \\
    \hline
  \end{tabular}
  \caption{Comparison of the conditioning of the stiffness matrix and number of iterations needed by conjugate gradients after preconditioning via Circulant matrices for $ r = 3 $. Nonconstant coefficient.} \label{tab:nonconstdeg3}
\end{table}

\begin{table}[h]
  \centering
  \begin{tabular}{|c|c|c|c|c|}
    \hline
    \multicolumn{1}{|c|}{} & \multicolumn{2}{|c|}{Conditioning} & \multicolumn{2}{|c|}{Iterations} \\
    \hline
    elements & Lagrangian & weights & Lagrangian & weights \\
    \hline
    $ 2 $ & $ 1.34 \times 10^2 $ & $ 8.56 \times 10^1 $ & $ 8 $ & $ 8 $ \\
    $ 4 $ & $ 6.05 \times 10^2 $ & $ 4.29 \times 10^2 $ & $ 16 $ & $ 12 $ \\
    $ 8 $ & $ 2.58 \times 10^3 $ & $ 1.95 \times 10^3 $ & $ 22 $ & $ 13 $  \\
    $ 16 $ & $ 1.07 \times 10^4 $ & $ 8.33 \times 10^3 $ & $ 23 $ & $ 13 $ \\
    $ 32 $ & $ 4.44 \times 10^4 $ & $ 3.44 \times 10^4 $ & $ 23 $ & $ 15 $ \\
    $ 64 $ & $ 1.82 \times 10^5 $ & $ 1.40 \times 10^5 $ & $ 25 $ & $ 15 $ \\
    $ 128 $ & $ 7.36 \times 10^5 $ & $ 5.65 \times 10^5 $ & $ 27 $ & $ 16 $ \\
    $ 256 $ & $ 2.96 \times 10^6 $ & $ 2.28 \times 10^6 $ & $ 28 $ & $ 18 $ \\
    \hline
  \end{tabular}
  \caption{Comparison of the conditioning of the stiffness matrix and number of iterations needed by conjugate gradients after preconditioning via Circulant matrices for $ r = 4 $. Nonconstant coefficient.} \label{tab:nonconstdeg4}
\end{table}

Table \ref{tab:nonconstdeg3} reports the performance of the preconditioned conjugated gradient for the degree $ r = 3 $, whereas Table \ref{tab:nonconstdeg4} shows the performance associated for $ r = 4 $. It is worth noting that the number of iterations, in the case of weights, quickly stabilizes. Moreover, for $ r = 4 $ it remains significantly under that observed for the equidistributed Lagrangian case. The impact of this in terms of runtime is depicted in Figure \ref{fig:timenonconst} (left, $r=3$ and right, $r=4$). 

\begin{figure}[h]
  \centering
  \includegraphics[width = 0.49\textwidth]{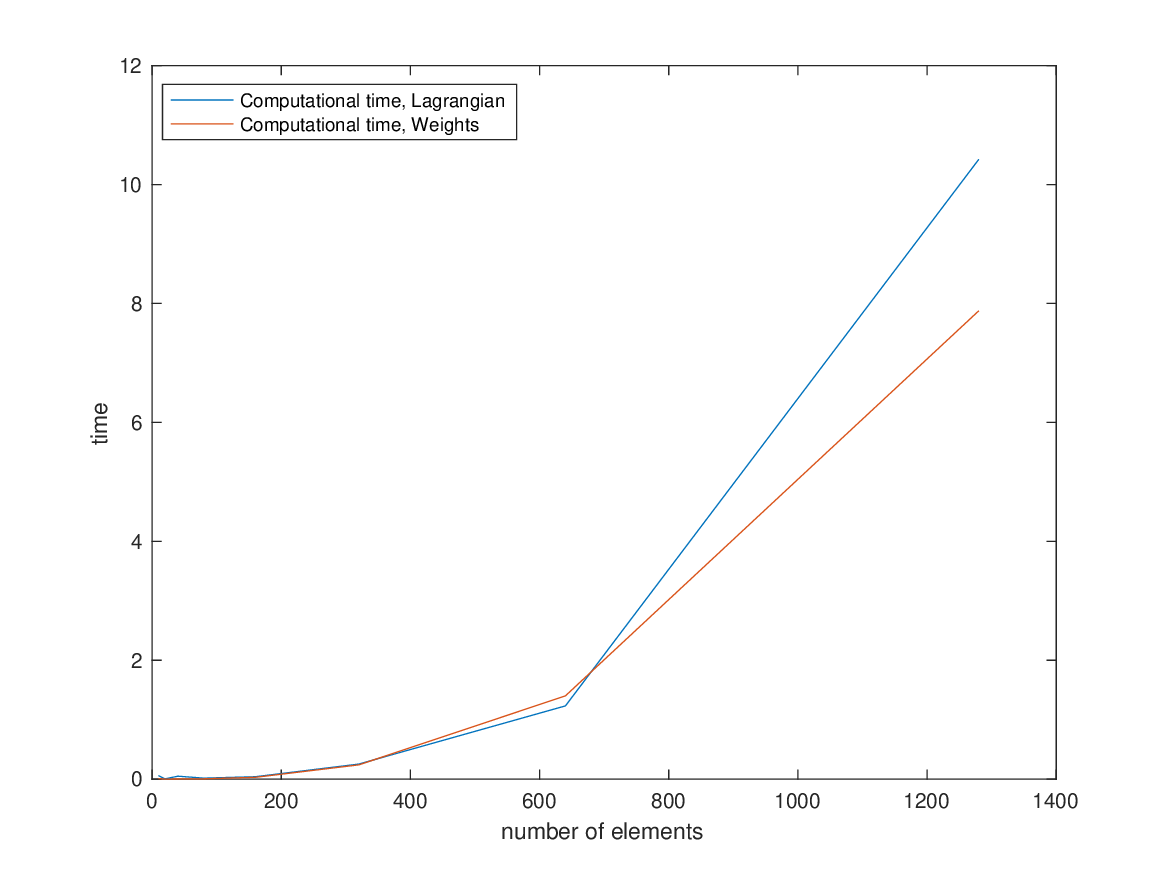}
  \includegraphics[width = 0.49\textwidth]{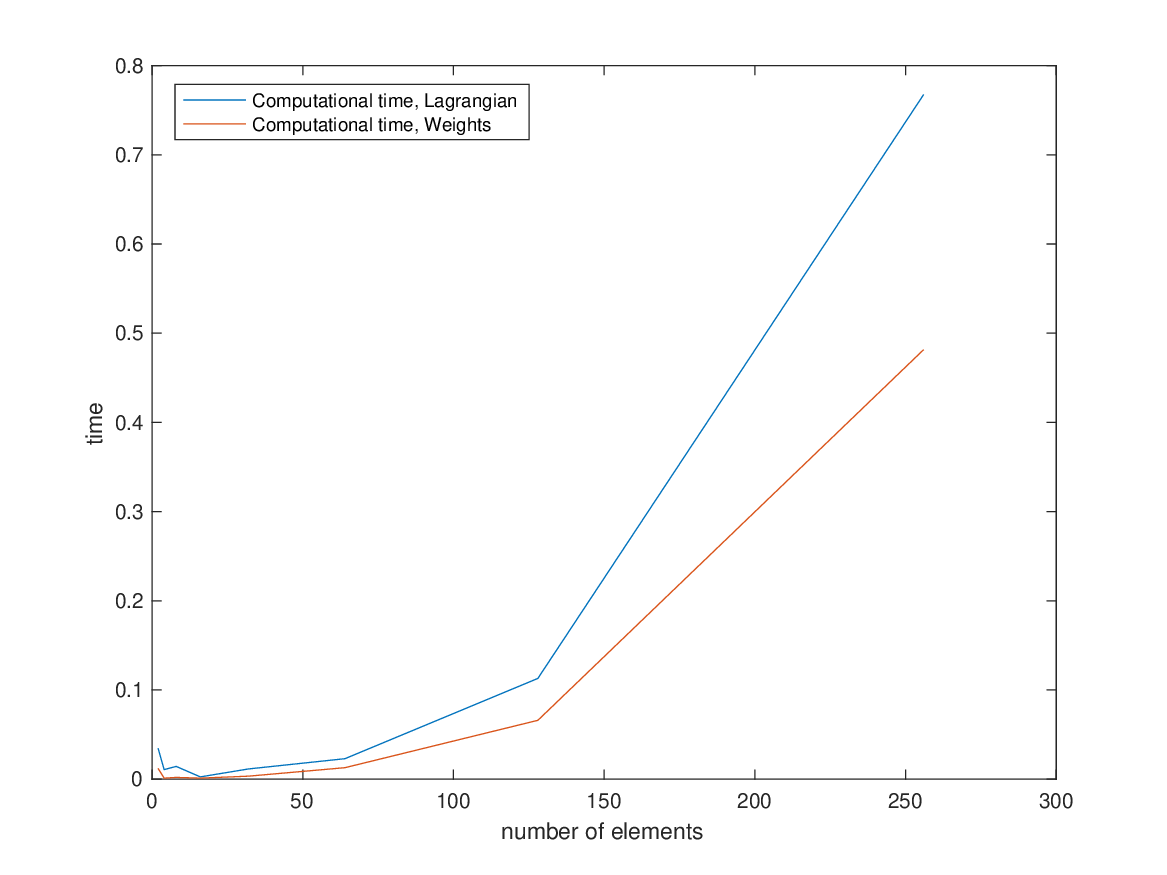}
  \caption{A comparison between equidistributed Lagrangian and optimised weights in terms of runtime in presence of a nonconstant coefficient.}
  \label{fig:timenonconst}
\end{figure}

\subsection{Graded mesh} \label{sect:quasiregularmesh}
We preferred to keep distinguished the case of nonconstant coefficients treated in Section \ref{sec:expnonconstcoeff} in order to better detect the effect of the mesh mapping; in this section we thus consider the elliptic problem with unit diffusion coefficient.
Let $g(x)$ be the mesh mapping function.
Following the remarks in Section~\ref{ssec:fr:quasiregular}, a preconditioner for the case of graded meshes can be obtained by setting $b(x)=\frac{1}{g'(x)}$ in \eqref{eq:Prec:b}.

In particular, we choose the mesh mapping $ g (x) = c e^{x} + k $, with $ c,k $ such that $ g(0) = 0 $ and $ g(1) = 1 $. 
Results are in accordance with the previous case and show an improvement in the case of optimized weights. A comparison in terms of iterations is depicted in Figure \ref{fig:varyingmesh}.

\begin{figure}[!h]
  \centering
  \includegraphics[width = 0.49\textwidth]{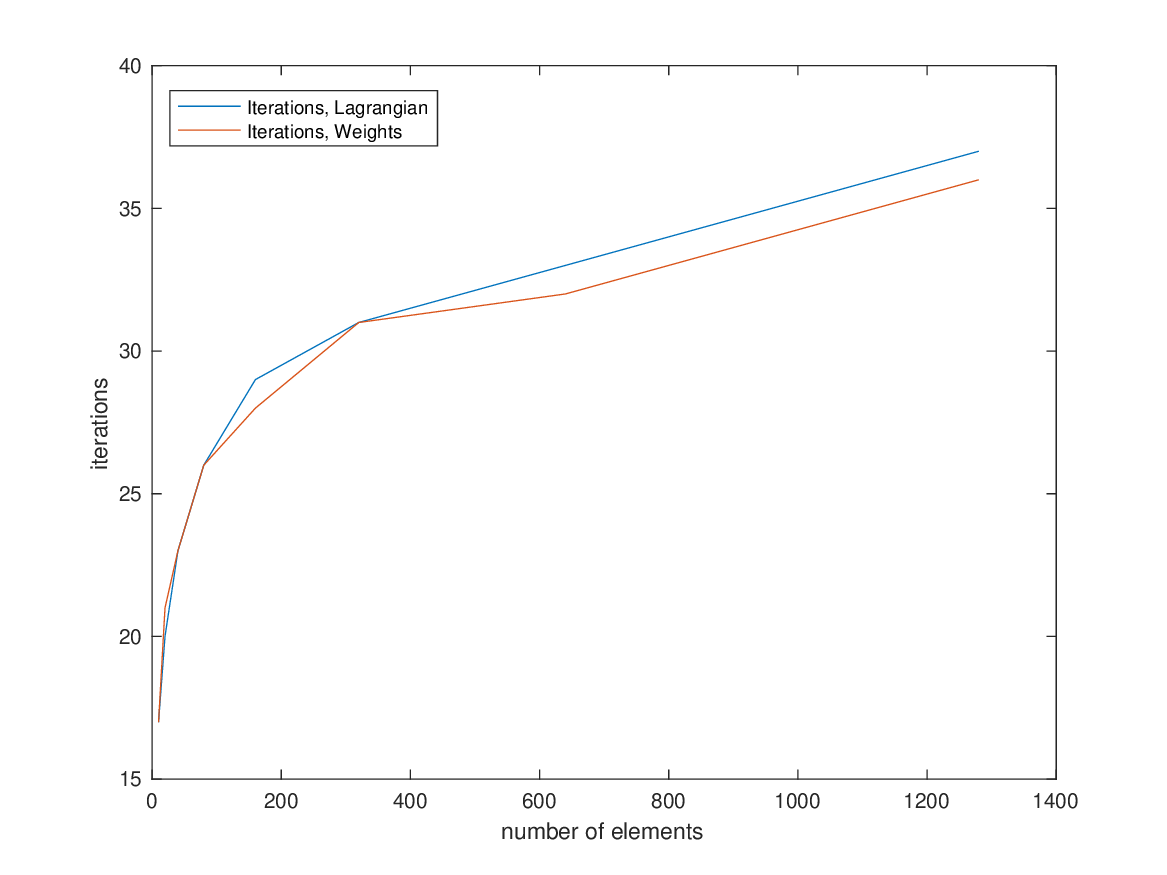} \includegraphics[width = 0.49\textwidth]{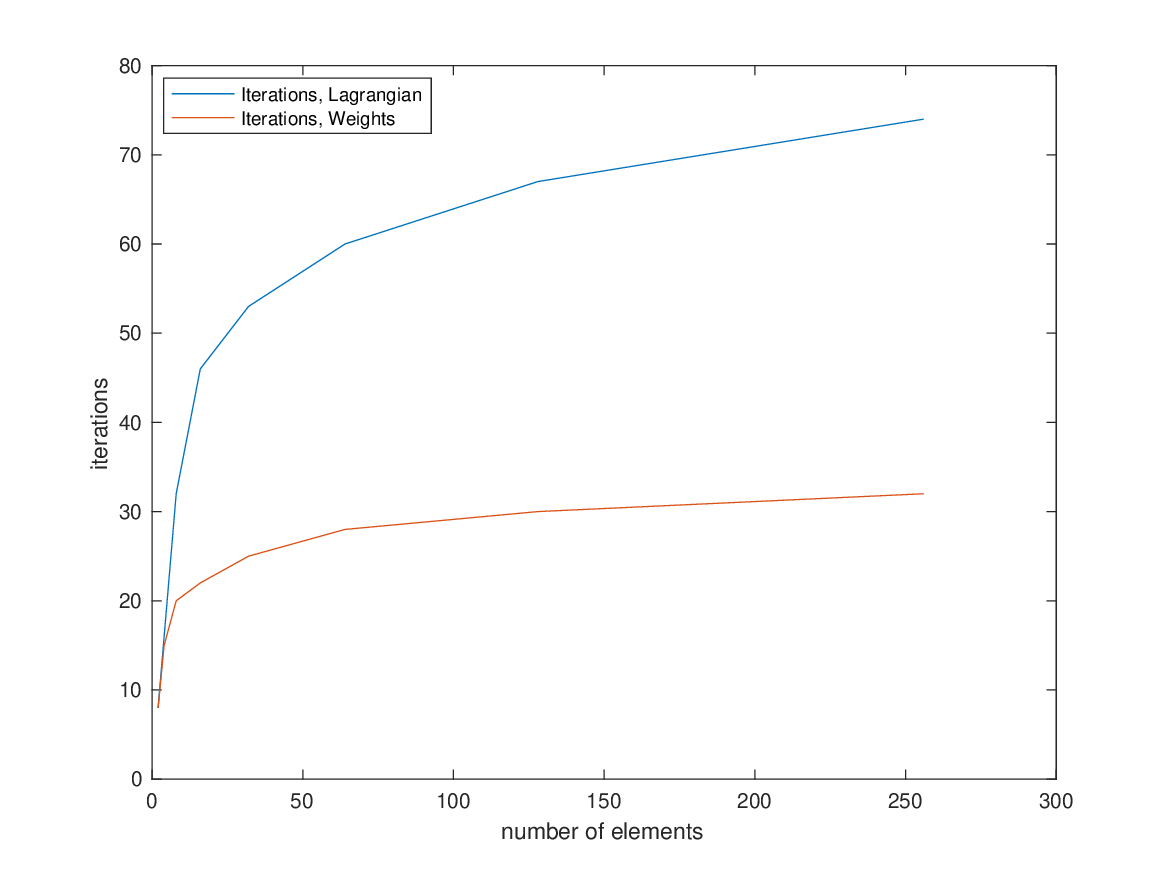}
  \caption{Mapped mesh. A comparison between equispaced Lagrangian and optimized weights in terms of PCG iterations for degrees $ r = 3 $ (left) and $ r = 4 $ (right).}
  \label{fig:varyingmesh}
\end{figure}

\subsection{Robustness of the preconditioner}

In this section we remove the hypothesis of uniformity or gradedness of the mesh. We in fact introduce in the mesh generator a randomness parameter $ \theta $. As we shall see, although the number of iterations and the runtime increase, it is evident that optimized weights improve both on the side of the number of iterations and on the side of the runtime.

The randomisation of the mesh is obtained as follows. First, we compute a uniform partitioning of the desired interval, say $ [0,1] $, and save the endpoints of the segments increasingly
$$ 0 = \xi_0 < \xi_1 < \ldots < \xi_n = 1 $$
with $ \xi_j = \frac{j}{n} $. Thus every segment $ I_j \doteq [\xi_{j-1}, \xi_j] $ has length $ |I_j | = \frac{j+1-j}{n} = \frac{1}{n} $. For each $ j = 1, \ldots, n-1 $ we create a random coefficient $ c_j \in \left( -\frac{1}{2n}, \frac{1}{2n} \right) $ and replace $ \xi_{j} $ with $ $ $\xi_i' = \xi_j + \theta c_j $. By construction, we obtain a sequence
$$ 0 = \xi_0' \doteq \xi_0 < \xi_1' < \ldots < \xi_n' \doteq \xi_n = 1 $$
and the desired randomised mesh is thus that whose $j$-th element is $ I_j' = [\xi_{j-1}', \xi_j'] $.

We propose the following numerical test to estimate the robustness of the preconditioner. We compare the performance, in terms of iterations, as the parameter $ \theta $ varies. In particular, we choose $ \theta $ increasing from $ 0.05 $ to $ 0.5 $ and report results in Figure \ref{fig:deg3randomMesh}.
\begin{figure}[h]
  \centering
  \includegraphics[width = 0.49\textwidth]{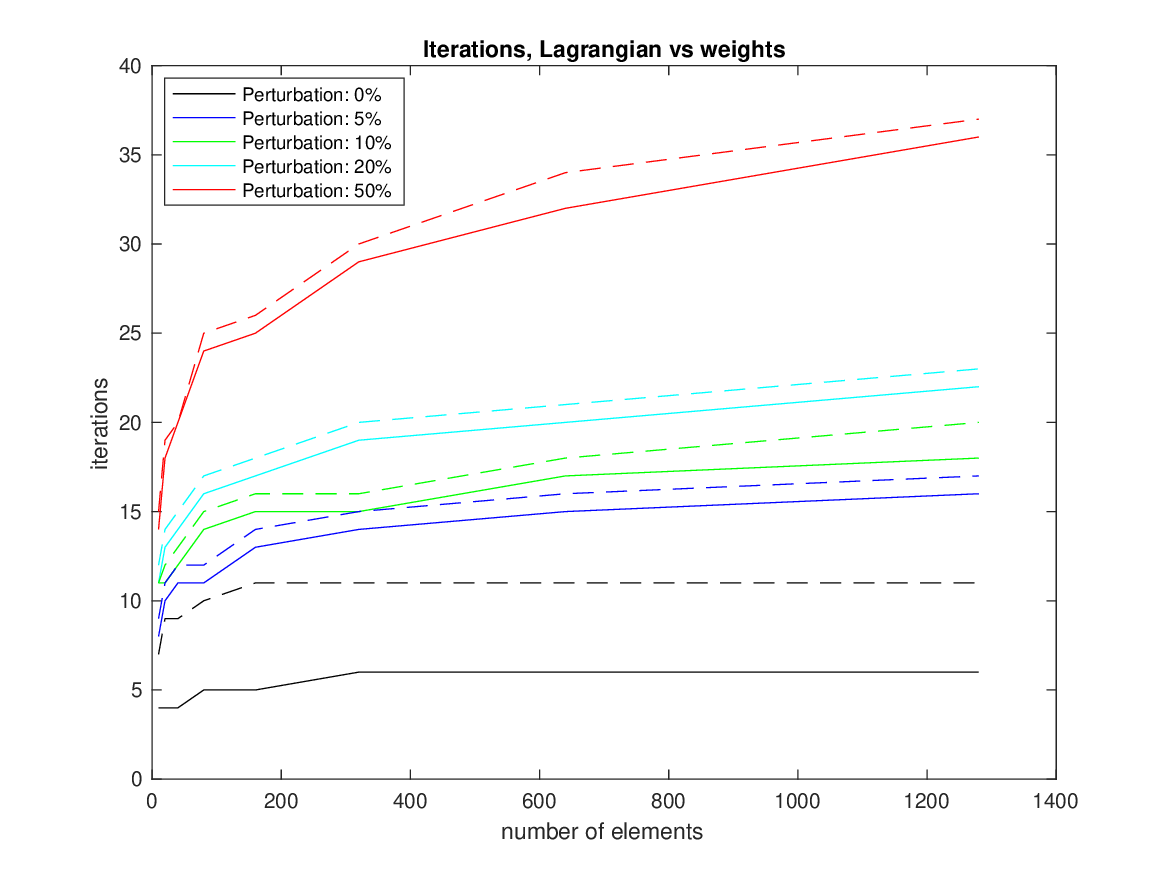} \includegraphics[width = 0.49\textwidth]{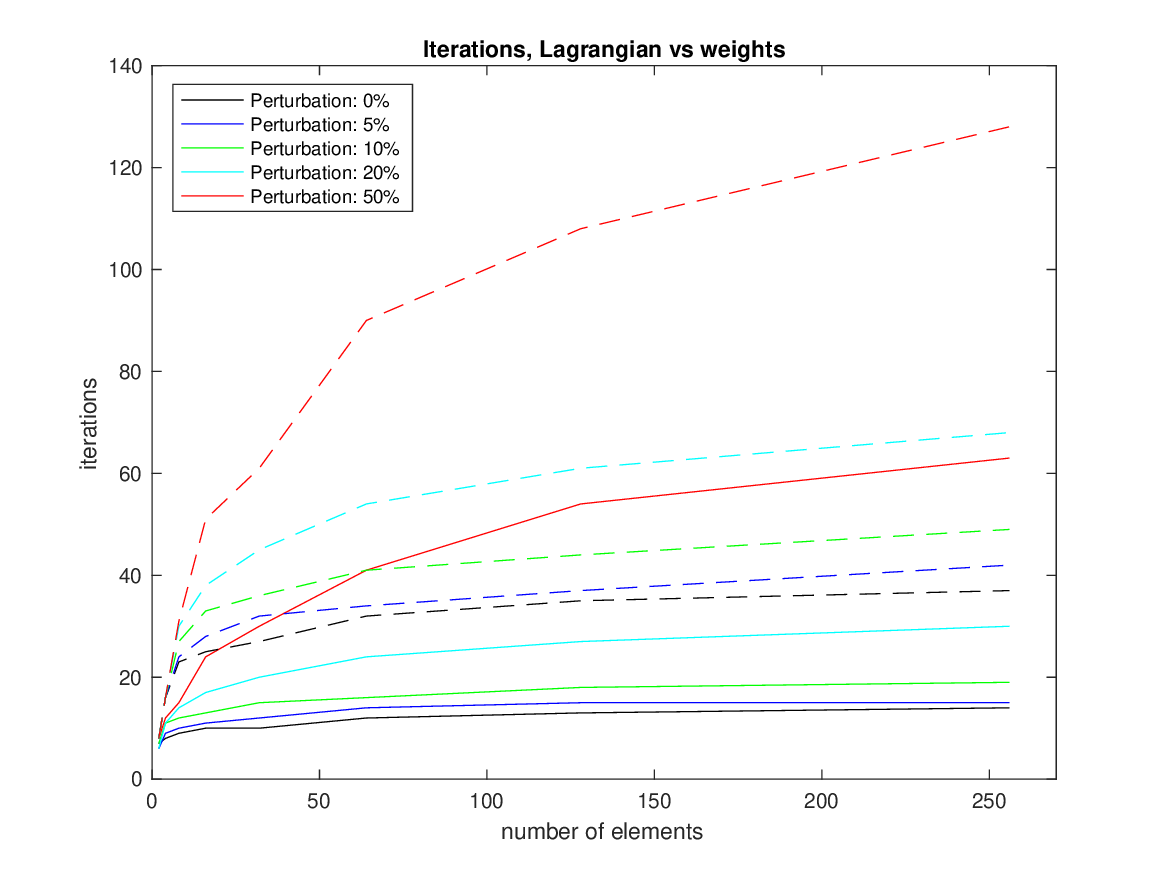}
  \caption{Randomised mesh. A comparison between equispaced Lagrangian and optimized weights in terms of PCG iterations for degrees $ r = 3 $ (left) and $ r = 4 $ (right). To same color corresponds the same entity of perturbation. Dashed lines represent equidistributed Lagrangian whereas solid lines represent optimized weights.}
  \label{fig:deg3randomMesh}
\end{figure}

We solve the linear system with the preconditioned conjugate gradient method, using the Circulant preconditioner built on the symbol $f^r_{\xi}$ which was computed for the uniform mesh case.
Clearly the analysis developed for the uniform case does not apply directly, as measures $ |I_j'|$'s may sensibly vary when increasing the parameter $\theta$.
Nevertheless, Figure \ref{fig:deg3randomMesh} shows that, up to a $50\%$ perturbation in the grid size ($\theta=0.5$),
the optimized weights perform better than the equidistributed Lagrangian (solid lines in the graph are always below the corresponding dashed ones) and that the preconditioner preserves some optimality properties (all lines show a ``plateaux'' for large grid sizes). The effects above are more pronounced for degree 4 discretizations, but visible also in the lower degree case.

\section{Conclusions}\label{sec:end}
In the current work we considered the weight generalization of Lagrangian finite elements of a standard onedimensional Laplacian. We have identified the formal expression of the matrix-sequences for the stiffness matrix of an elliptic test problem, in terms of the mesh size and problem/discretization variables.

Using the GLT theory, we have then provided a detailed spectral analysis of the matrix-sequence by emphasizing the spectral distribution in the Weyl sense, including the branch analysis and the clustering. We have also studied the extremal spectral behavior, which allows to identify the asymptotic conditioning in spectral norm, as the matrix order tends to infinity i.e. as the mesh size tends to zero. General techniques and formulae have been provided and the spectral analysis has been then used for identifying ad hoc preconditioners, in dependence of the discretization parameters and in connection with the preconditioned conjugate gradient method.
A spectral clustering at one has been proven for the resulting preconditioned matrix-sequences, also in the variable coefficient setting, which agrees with the observed optimal and robust performance of the used preconditioning technique.
Numerical visualization and experimental tests have been performed and critically discussed.
In view of extensions to higher dimensional cases, in which using regular meshes is harder, we have also tested with good results our preconditioning technique on randomly perturbed meshes.

As future steps, it is for sure of interest to consider the Laplacian in $d>1$ dimensions and more involved PDEs, knowing that the spectral multilevel block Toeplitz \cite{ty-1,tilli-L1} and the spectral multilevel block GLT tools are available in that setting \cite{GLTbookII,etnaGLTbookIV,barbarino-reducedGLT} too, also for dealing with non Cartesian domains. However, we should be aware that the preconditioning is hardly effective alone in the multidimensional case owing to the theoretical barriers proved in \cite{Tyrty1,Tyrty2,nega-gen}, and a multi-iterative approach has to be considered (see \cite{multi-galerkin,multi-colloc} and references therein), where the main underlying method is of multigrid type like in item {\bf a2.}, beginning of Section \ref{sec:numer}, and, as in the current work, the preconditioned conjugate gradient based on the GLT symbols is employed in the smoothing steps.
Moreover, the choice of the reference domain could lead to relevant difficulties \cite{BFZ23}, which have to be carefully handled. Nevertheless, because of the very general definition of weights in terms of differential forms, this kind of analysis can be applied and generalised to very large families of finite elements, helping thus in understanding and improving the conditioning associated with problems involving, for instance, Raviart-Thomas or N\'ed\'elec elements (see, e.g., \cite{Ainsworth} for a valuable perspective on the latter problem).

Discontinuous elements are suitable for this kind of analysis as well. However, substantial technical changes have to be introduced: for instance, Proposition \ref{prop:computingVanderm} and the related derived results shall be adapted to the new setting. Numerical results in that direction are promising and seem to confirm the fact that the indicated difficulties represent in fact technical aspects for the theoretical analysis, rather than insurmountable obstacles.

\section*{Acknowledgements and funding} The work of all authors is supported by the Italian Agency INdAM-GNCS. \\
The work of Ludovico Bruni Bruno is funded by INdAM and supported by Padova University. This author expresses his gratitude to Insubria University, Junior Postdoc project 2022-2023 leaded by Matteo Semplice, on whose framework the majority of this research has been developed.
\\
The work of Stefano Serra-Capizzano is funded from the European High-Performance Computing Joint Undertaking  (JU) under grant agreement No 955701. The JU receives support from the European Union’s Horizon 2020 research and innovation programme and Belgium, France, Germany, Switzerland.
Furthermore Stefano Serra-Capizzano is grateful for the support of the Laboratory of Theory, Economics and Systems – Department of Computer Science at Athens University of Economics and Business.

\bibliography{SmallSimplices}
%
%
%
%

\end{document}